\documentclass[12pt,a4paper]{article}
\usepackage[utf8]{inputenc}
\usepackage{amsmath, amssymb, amsthm}
\usepackage{bbm}
\usepackage{geometry}
\usepackage{xcolor}
\usepackage{graphicx}
\usepackage{float}
\usepackage{subcaption}  
\usepackage{comment}

\newtheorem{theorem}{Theorem}[section]
\newtheorem{proposition}[theorem]{Proposition}
\newtheorem{lemma}[theorem]{Lemma}

\newtheorem{definition}[theorem]{Definition}
\newtheorem{remark}[theorem]{Remark}

\newcommand{\R}{\mathbb{R}}

\definecolor{verde}{RGB}{0,150,0}

\title{The pointwise multiscale texture operator: analytical foundations and functional characterization}
\author{Carlos Fernández García \& Zulima Fernández Mu\~niz}
\date{\today}

\begin{document}
	
	\maketitle
	
	\begin{abstract}

	Texture is commonly described through statistical descriptors, filter-bank responses, or multiscale representations, yet its functional-analytic structure remains largely unexplored. In this work we investigate texture from the viewpoint of Gaussian scale-space evolution, interpreting it as the persistence of local structures across scales. This perspective naturally leads to a multiscale texture operator that quantifies the evolution of local image structures under Gaussian diffusion.
	We establish the fundamental analytical properties of this operator, proving its boundedness on $L^p(\mathbb{R}^n)$, its infinite-order smoothing effect in the Hilbert setting, and an explicit spectral characterization showing that it acts as a localized Gaussian band-pass filter. Motivated by the logarithmic organization of Gaussian scales, we introduce an $r$-adic Difference-of-Gaussians discretization and prove Littlewood--Paley-type stability estimates together with an explicit reconstruction formula.
	Building on this multiscale representation, we define multiscale texture norms that quantify the persistence of local structures across scales. These norms induce a natural family of Banach spaces, with a distinguished Hilbertian case, and we prove that they provide an equivalent characterization of the classical Besov and Sobolev scales. Consequently, classical functional regularity can be described intrinsically through the persistence of texture under Gaussian scale evolution.
	These results establish the first functional-analytic characterization of texture persistence under Gaussian scale evolution, revealing that classical Besov and Sobolev regularity admit an equivalent description in terms of multiscale texture persistence. This provides a rigorous bridge between Gaussian scale-space theory, harmonic analysis, and variational models based on texture.
	
\end{abstract}
	
	\section{Introduction}
	
		When we observe a function or an image, we are able to extract a remarkable amount of geometric and structural information. This naturally raises a fundamental question: where is this information actually encoded?
		
		A first answer would be to claim that the information is contained in the pointwise values of the function itself. Experience in image analysis, however, has shown that this viewpoint is often insufficient: two images whose values differ only slightly may represent completely different objects, while images with very different intensities may share essentially the same geometric structure.
		
		This observation led to a change of perspective in which derivatives, and particularly the gradient, acquired a central role. Edges, discontinuities, and sharp intensity transitions contain a substantial part of the visual and geometric information carried by an image, and much of modern image processing theory is built upon this principle.
		
		Nevertheless, even a description based on derivatives is often unable to capture certain structural phenomena. In many applications, the relevant information is not associated solely with edges or discontinuities, but also with oscillatory patterns, local regularities, and spatial organizations that are commonly referred to collectively as texture.
		
		The central idea underlying this work is that the evolution of local structures across scales carries information that is not accessible from pointwise values or spatial derivatives alone. We propose to interpret texture through the persistence of contextual structures under Gaussian scale evolution, thereby introducing a multiscale representation in which texture is viewed as a manifestation of structural persistence rather than as a collection of local statistical or frequency-based features.
		
		This perspective naturally leads to a family of operators associated with Gaussian scale evolution and, more importantly, to a functional framework in which persistence becomes the fundamental analytical quantity. As will be shown throughout the paper, this viewpoint establishes direct connections with harmonic analysis, Littlewood--Paley theory, and the classical Besov and Sobolev scales, while providing a mathematically rigorous interpretation of multiscale texture persistence.
		
		\subsection{Motivation from texture analysis and imaging}
		
			Texture is a core structural component in image analysis. Together with shape, color, and motion, it is one of the main visual cues that allow a human or artificial observer to segment, recognize, and interpret a scene. From the fine textures of biological tissues in medical imaging to the regular patterns of industrial surfaces or the subtle variations appearing in remote sensing data, texture information plays an important role in many applications.
			
			Despite its importance, texture has no universally accepted mathematical definition. Intuitively, a texture may be understood as a local repetition of an elementary pattern, possibly accompanied by gradual variations in orientation, scale, or contrast. However, this intuitive notion is hard to formalize. What exactly constitutes a pattern? How should repetition be measured? Over which range of scales should this repetition occur?
			
			Many approaches have been proposed over the past decades for texture analysis. Statistical methods describe texture through local distributions and co-occurrence patterns, filter-bank methods rely on multiscale frequency decompositions, and wavelet methods provide efficient multiresolution representations. More recently, deep learning techniques have shown strong performance in texture classification and segmentation tasks.
			
			Despite their differences, most existing approaches share a common feature: they describe texture either through local statistics or through responses to families of filters. While highly successful in practice, these methods often give little insight into why a structure is perceived as a texture, and their link to classical regularity notions is often unclear.
			
			This leads naturally to the question: can texture be described through the persistence of structures across scales?
			
			The approach developed here builds on the observation that textures are neither purely local nor purely global objects. Rather, they emerge from the way local contextual structures evolve under changes of observation scale. Persistent structures generate slowly varying multiscale trajectories, whereas unstable oscillatory configurations disappear rapidly under Gaussian diffusion. In this sense, texture is interpreted as the observable manifestation of persistence across scales rather than as a purely statistical or frequency-based phenomenon.
			
			This perspective naturally leads to the study of Gaussian scale-space evolution, where each point is represented by its entire scale trajectory rather than by isolated responses at prescribed resolutions. The remainder of the paper develops the analytical framework generated by this evolution, including the associated persistence operator, its spectral properties, its $r$-adic multiscale decomposition, and the functional structures naturally induced by these multiscale trajectories.
			
			From the point of view of harmonic analysis and coorbit theory, the multiscale decomposition induced by $\tau$ can be placed within the general framework of Besov–Lizorkin–Triebel spaces as developed by Ullrich and others; this connection will be made explicit in Section~\ref{sec:spaces}.
			
		\subsection{The pointwise multiscale texture operator and its relation to existing frameworks}
		
			The construction proposed in this work originates from Gaussian scale-space theory \cite{witkin,koenderink,lindeberg}, a widely studied multiscale representation in image analysis. Given a function $u:\mathbb{R}^n\to\mathbb{R}$, we define its scale-space evolution by
			\begin{equation*}
				U(z,\sigma) = (u*\gamma_\sigma)(z) = \int_{\mathbb{R}^n}  u(y)\,\gamma_\sigma(z-y)\,dy,
			\end{equation*}
			where $\gamma_\sigma$ denotes the Gaussian kernel with variance $\sigma^2$.
			
			For each point $z$ in the domain, the function $U(z,\cdot)$ describes how the local Gaussian average of $u$ around $z$ evolves as the observation scale changes. This one-parameter family, known as the Gaussian scale-space representation \cite{witkin,koenderink,lindeberg}, satisfies fundamental properties such as linearity, isotropy, causality, and the diffusion equation
			\begin{equation*}
				\partial_\sigma U = \sigma\Delta U.
			\end{equation*}
			
			Rather than regarding the scale parameter merely as an external observation parameter, we treat it as an additional analytical variable. Consequently, each point is no longer represented by a single value but by an entire trajectory in Gaussian scale-space, describing the evolution of its local contextual identity across scales.
			
			The trajectory $U(z,\cdot)$ contains the complete multiscale evolution of the local contextual information around the point $z$. Our interest is not only the trajectory itself but also the rate at which it evolves as the observation scale changes. This naturally leads to the definition of the pointwise multiscale texture operator
			\begin{equation*}
				\tau_{z,u}(\sigma) = \partial_\sigma U(z,\sigma).
			\end{equation*}
			
			Since $\tau_{z,u}(\sigma)$ is the derivative of the contextual multiscale signature with respect to the diffusion scale, it measures the instantaneous rate at which the local contextual identity changes under Gaussian evolution. Structures whose contextual organization remains essentially unchanged across scales generate small responses, whereas structures that rapidly disappear or reorganize under diffusion produce larger values of $\tau$. In this sense, the operator provides a quantitative measure of multiscale persistence, while texture is interpreted as the observable manifestation of this persistence across scales.
			
			\medskip\noindent\textbf{Relation with scale-space theory and diffusion.}
			
			Using the diffusion identity
			\begin{equation*}
				\partial_\sigma\gamma_\sigma = \sigma\Delta\gamma_\sigma,
			\end{equation*}
			we obtain the equivalent representation
			\begin{equation*}
				\tau_{z,u}(\sigma) = (u*\Phi_\sigma)(z), \qquad \Phi_\sigma =  \sigma\Delta\gamma_\sigma.
			\end{equation*}
			
			Consequently, the texture operator can be viewed as a convolution with an explicit kernel of zero mean. Although each fixed-scale response admits this convolutional representation, the object studied throughout this paper is the complete scale-dependent family generated by Gaussian evolution rather than isolated responses at prescribed scales. As will be shown later, the kernel acts as a localized band-pass filter and provides a natural mechanism for detecting oscillatory structures at their intrinsic scales.
			
			\medskip\noindent\textbf{Relation with DoG and LoG filters.}
			
			For each fixed observation scale, the operator is closely related to the classical Difference-of-Gaussians (DoG) and Laplacian-of-Gaussian (LoG) filters \cite{lindeberg,mallat}.
			
			Indeed, for nearby scales one has
			\begin{equation*}
				U(z,\sigma_2)-U(z,\sigma_1) \approx \tau_{z,u}(\sigma)\,(\sigma_2-\sigma_1),
			\end{equation*} 
			showing that $\tau$ may be interpreted as the continuous-scale limit of a family of DoG filters. This establishes a direct analytical connection with classical multiscale filtering. However, unlike DoG or LoG methods, which are typically employed through responses at selected scales, the present framework studies the complete evolution of the scale derivative as a function of the diffusion parameter. In this sense, the operator provides an analytically well-defined realization of a widely used multiscale construction while extending it to a genuinely functional setting.
			
			\medskip\noindent\textbf{Relation with multiscale decompositions.}
			
			A geometric discretization of the scale variable with ratio $r>1$ leads naturally to a family of Difference-of-Gaussians filters satisfying dilation relations and constant relative bandwidth properties \cite{mallat}. Beyond its computational convenience, this discretization provides a natural bridge between continuous Gaussian scale evolution and classical multiresolution analysis. As will be shown later, the resulting system yields stable multiscale representations together with reconstruction formulas and Littlewood--Paley-type estimates, thereby connecting texture persistence with the analytical framework of harmonic analysis.
			
			\medskip\noindent\textbf{Relation with previous work.}
			
			A first systematic use of the operator appeared in \cite{fernandez2}, where $\tau$ was introduced as a pointwise multiscale texture descriptor and successfully applied to the segmentation of 3D medical images. That work established the practical relevance of the operator, while leaving open the analytical questions concerning its functional, spectral, and multiscale structure.
			
			More recently, the diffusion--persistence variational model proposed in \cite{fernandez1} incorporated $\tau$ into a variational framework for texture-preserving image restoration, showing that multiscale persistence can also serve as a regularization principle. The emphasis there was placed on the variational model and its numerical behaviour rather than on the mathematical analysis of the operator itself.
			
			The purpose of the present article is to develop the mathematical foundations underlying these previous constructions. More precisely, we study the analytical properties of the texture operator, establish its spectral characterization, derive its associated $r$-adic multiscale decomposition, and investigate the functional structures naturally induced by this representation. A central result of the paper is that the $r$-adic DoG decomposition gives rise to a family of multiscale texture spaces $\mathcal T^{s,p,q}(\mathbb R^n)$ that fit naturally within the  classical Besov scale and, in the Hilbertian case, provide an equivalent realization of Sobolev regularity expressed entirely through texture responses across scales.

		\subsection{Main contributions}
		
			The present article is devoted to a systematic analytical study of the pointwise multiscale texture operator $\tau_{z,u}(\sigma)$ and of the multiscale structures induced by its associated Gaussian scale-space representation. The main contributions of the paper can be summarized as follows.
			
			\medskip
			
			\noindent\textbf{1. Analytical properties of the texture operator.}
			
			We establish boundedness and regularization properties for the operator $\tau$. In particular, for every fixed scale $\sigma>0$, the mapping
			\begin{equation*}
				u \mapsto \tau_{\cdot,u}(\sigma)
			\end{equation*}
			is shown to be bounded on $L^p(\mathbb R^n)$ for all $1\le p\le\infty$. In the Hilbert setting, the operator exhibits an infinite-order smoothing effect, mapping $L^2(\mathbb R^n)$ continuously into $H^k(\mathbb R^n)$ for every $k\ge0$.
			
			\medskip
			
			\noindent\textbf{2. Spectral characterization and band-pass behavior.}
			
			We derive an explicit expression for the Fourier symbol of the kernel associated with $\tau$ and analyze its spectral properties. The resulting operator acts as a localized band-pass filter whose characteristic frequency is proportional to the inverse observation scale and whose symbol decays exponentially at high frequencies. This provides a mathematical explanation of the ability of $\tau$ to detect oscillatory structures at their intrinsic scales.
			
			\medskip
			
			\noindent\textbf{3. $r$-adic discretization and multiscale decomposition.}
			
			We introduce a geometric discretization of the scale variable through an $r$-adic progression and study the associated Difference-of-Gaussians family. The resulting system satisfies dilation relations, possesses constant relative bandwidth, and yields stable multiscale representations together with reconstruction and Littlewood--Paley-type estimates in $L^2(\mathbb R^n)$.
			
			\medskip
			
			\noindent\textbf{4. Hilbert texture spaces and connection with Sobolev regularity.}
			
			Motivated by the previous multiscale analysis, we introduce the Hilbert texture spaces $\mathcal T^{s}(\mathbb R^n)$ and characterize them through weighted texture coefficients. We prove that these spaces are Hilbert spaces and identify them with the classical Sobolev spaces $H^s(\mathbb R^n)$ up to equivalence of norms. This characterization provides an alternative realization of Sobolev regularity entirely expressed in terms of multiscale texture responses.
			
			\medskip
			
			\noindent\textbf{5. A functional framework for texture persistence.}
			
			Taken together, the previous results provide a rigorous mathematical framework for the notion of texture persistence across scales. They establish direct connections between Gaussian scale-space theory, multiscale texture analysis, harmonic analysis, and classical regularity theory, thereby supplying the analytical foundations underlying the persistence-based methodologies introduced in \cite{fernandez1,fernandez2}.
			
			\medskip
			
			\noindent
			In this sense, the present article provides the mathematical foundations of the pointwise multiscale texture operator introduced in \cite{fernandez2}. Beyond the individual analytical results, it shows that texture persistence admits a natural formulation within Gaussian scale-space theory and classical harmonic analysis.
			
			The logical development of the paper is summarized in Figure~\ref{fig:pipeline}.
			
			\begin{figure}[H]
				\centering
				\includegraphics[width=0.65\textwidth]{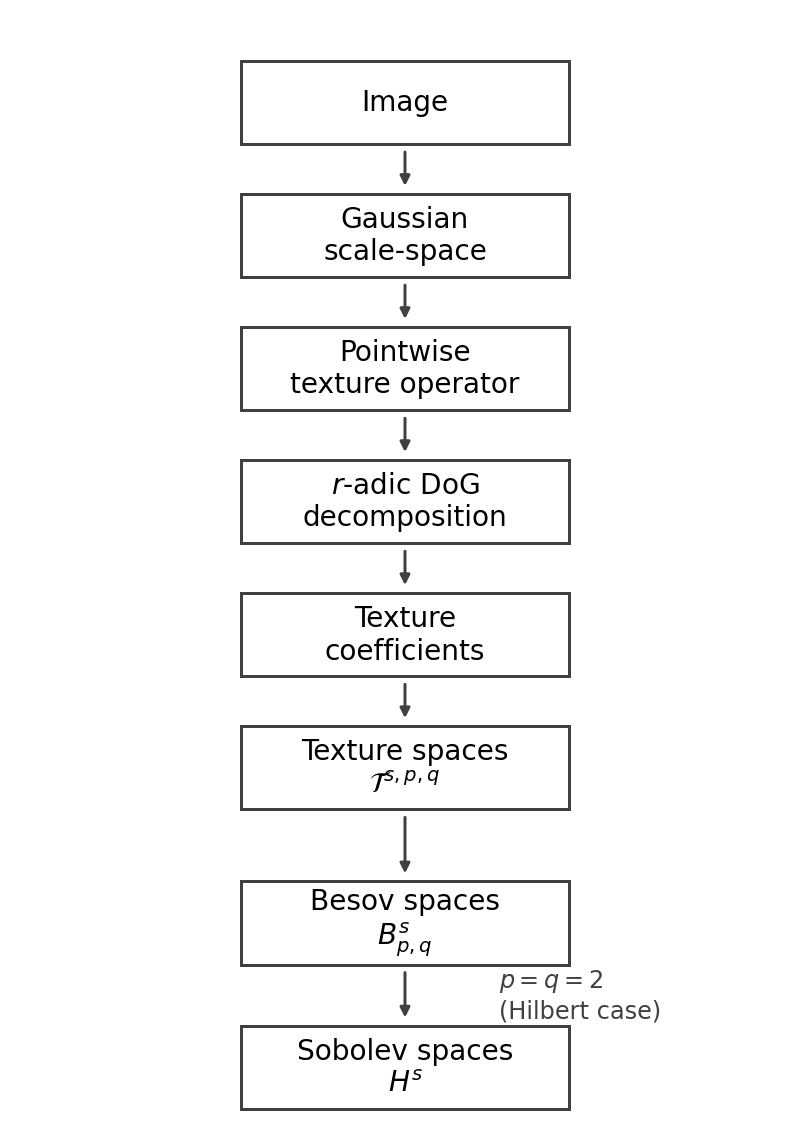}
				\caption{Construction of the multiscale texture framework developed in this paper. Starting from the Gaussian scale-space representation, the pointwise texture operator induces an $r$-adic Difference-of-Gaussians decomposition, whose texture coefficients define the multiscale texture spaces $\mathcal{T}^{s,p,q}(\mathbb{R}^n)$. This space admit a description in terms of the classical Besov scale, while preserving a geometric interpretation in terms of texture persistence. In the Hilbertian case $(p=q=2)$, this characterization recovers the Sobolev spaces.}
				\label{fig:pipeline}
			\end{figure}
		
		\subsection{Organization of the paper}
		
			The remainder of the paper is organized as follows. In Section~\ref{sec:preliminares} we introduce the notation and present a number of illustrative examples of the texture operator. Section~\ref{sec:basicproperties} is devoted to its basic analytical properties, including boundedness and regularization results. In Section~\ref{sec:spectral} we study its spectral representation and band-pass behavior. Section~\ref{sec:radic} introduces the $r$-adic discretization and the associated multiscale decomposition. Section~\ref{sec:spaces} develops the Hilbert texture spaces associated with the multiscale decomposition and establishes their equivalence with the classical Sobolev spaces. Finally, Section~\ref{sec:conclusions} contains conclusions and several directions for future research.
	
	\section{Preliminaries and motivating examples}\label{sec:preliminares}
	
		Before undertaking the functional-analytic study of the texture operator, we fix the notation used throughout the paper and present some simple examples that help illustrate the qualitative behavior of the operator.
		
		\subsection{Notation and basic functional-analytic tools}
		
			We denote by $\mathbb{R}^n$ the Euclidean space of dimension $n$, where $n=2$ for standard images, although the results are established for any $n\ge 1$. Spatial variables are denoted by $z, x, y \in \mathbb{R}^n$, while $\sigma > 0$ represents the scale parameter.
			
			For $1 \le p \le \infty$, $L^p(\mathbb{R}^n)$ is the space of measurable functions $u:\mathbb{R}^n\to\mathbb{R}$ such that
			\begin{equation*}
				\|u\|_{L^p} = \left( \int_{\mathbb{R}^n} |u(x)|^p dx \right)^{1/p} < \infty
			\end{equation*}
			(with the usual modification using the essential supremum for $p=\infty$).
			
			Given a kernel $\Phi \in L^1(\mathbb{R}^n)$, the convolution of $u$ with $\Phi$ is defined as
			\begin{equation*}
				(u * \Phi)(z) = \int_{\mathbb{R}^n} u(y) \Phi(z-y) \, dy.
			\end{equation*}
			Young's inequality states that for $1 \le p \le \infty$,
			\begin{equation*}
				\|u * \Phi\|_{L^p} \le \|u\|_{L^p} \|\Phi\|_{L^1}.
			\end{equation*}
			
			We denote by $\mathcal{S}(\mathbb{R}^n)$ the space of rapidly decreasing functions (Schwartz space) and by $\mathcal{S}'(\mathbb{R}^n)$ its dual, the space of tempered distributions. The Fourier transform is defined as
			\begin{equation*}
				\hat{u}(\xi) = \int_{\mathbb{R}^n} u(x) e^{-i x\cdot\xi} \, dx,
			\end{equation*}
			although in this paper we will use a normalization constant adapted to the expression of the Gaussian kernel.
			
			The Gaussian kernel (or heat kernel) in $\mathbb{R}^n$ is
			\begin{equation*}
				\gamma_{\sigma}(x) = \frac{1}{(2\pi\sigma^2)^{n/2}} \exp\left(-\frac{|x|^2}{2\sigma^2}\right), \quad \sigma > 0.
			\end{equation*}
			This kernel satisfies $\int_{\mathbb{R}^n} \gamma_{\sigma}(x) dx = 1$ and its Fourier transform is
			\begin{equation*}
				\hat{\gamma}_{\sigma}(\xi) = e^{-\frac{\sigma^2|\xi|^2}{2}}.
			\end{equation*}
			
			The Gaussian scale-space of a function $u \in L^1(\mathbb{R}^n)$ or, more generally, a tempered distribution $u \in \mathcal{S}'(\mathbb{R}^n)$, is defined as the family $\{U(\cdot,\sigma)\}_{\sigma>0}$ where
			\begin{equation*}
				U(z,\sigma) = (u * \gamma_{\sigma})(z) = \int_{\mathbb{R}^n} u(y) \gamma_{\sigma}(z-y)\,dy,
			\end{equation*}
			with the convolution understood in the distributional sense for $u \in \mathcal{S}'(\mathbb{R}^n)$. For each $\sigma > 0$, $U(\cdot,\sigma)$ is a smoothed version of $u$; as $\sigma \to 0$, one has $U(\cdot,\sigma) \to u$ almost everywhere (under suitable conditions).
		
		\subsection{Illustrative examples of $\tau$}
			
			Before developing the analytical theory, it is instructive to examine the behavior of the texture operator on a number of simple yet representative examples. Although elementary, these examples illustrate the principal qualitative behaviors of the operator under Gaussian scale evolution and anticipate many of the analytical properties established in the subsequent sections. Their purpose is not to provide rigorous proofs, but to build intuition for the mathematical theory developed later.
		
			Using the notation introduced in previous subsection, we write $g_{z,u}(\sigma) = U(z,\sigma)$ to emphasize its interpretation as a multiscale contextual signature. We recall the fundamental definitions. For $u \in L^1(\mathbb{R}^n)$ and $z \in \mathbb{R}^n$, we define the \emph{multiscale contextual signature} by
			\begin{equation*}
				g_{z,u}(\sigma) = (u * \gamma_{\sigma})(z) = \int_{\mathbb{R}^n} u(y) \gamma_{\sigma}(z-y)\,dy.
			\end{equation*}
			The function $\sigma \mapsto g_{z,u}(\sigma)$ describes how the local average of $u$ around the point $z$ evolves as the observation scale increases. For small $\sigma$, $g_{z,u}(\sigma)$ approximates the pointwise value $u(z)$; for large $\sigma$, it tends to the global mean value of $u$ (if the function is integrable).
			
			The \emph{pointwise multiscale texture operator} is defined as the derivative of this signature with respect to the scale:
			\begin{equation*}
				\tau_{z,u}(\sigma) = \frac{\partial}{\partial\sigma} g_{z,u}(\sigma).
			\end{equation*}
			
			Intuitively, $\tau_{z,u}(\sigma)$ measures the \emph{rate of change} of the local neighbourhood of $z$ as the scale varies. Small values of $|\tau_{z,u}(\sigma)|$ indicate that the structure around $z$ is relatively homogeneous at scale $\sigma$; large values indicate the presence of transitions, edges, or active textural patterns at that scale.
			
			\medskip
			
			\noindent \textbf{Example 1: Constant functions.}
			Let $u(x) = c$ be constant. Then $g_{z,u}(\sigma) = c$ for all $\sigma$, and therefore $\tau_{z,u}(\sigma) = 0$ for all $z$ and $\sigma$. The operator does not detect texture in the absence of local variations.
			
			\medskip
			
			\noindent \textbf{Example 2: An ideal edge (step function).}
			Consider in one dimension $u(x) = \mathbf{1}_{[0,\infty)}(x)$, the Heaviside step. For a point $z$ far from the edge, $g_{z,u}(\sigma)$ is essentially constant and $\tau_{z,u}(\sigma)$ is small. Near the edge, however, the signature $g_{z,u}(\sigma)$ smoothly transitions from $0$ to $1$ over an interval of order $\sigma$, and its derivative $\tau_{z,u}(\sigma)$ attains a maximum in the vicinity of the discontinuity. The maximum occurs at scales comparable with the distance to the discontinuity. The operator thus detects the edge, but does so through a smooth, scale-dependent signature.
			
			\begin{figure}[htb]
				\centering
				\includegraphics[width=.48\textwidth]{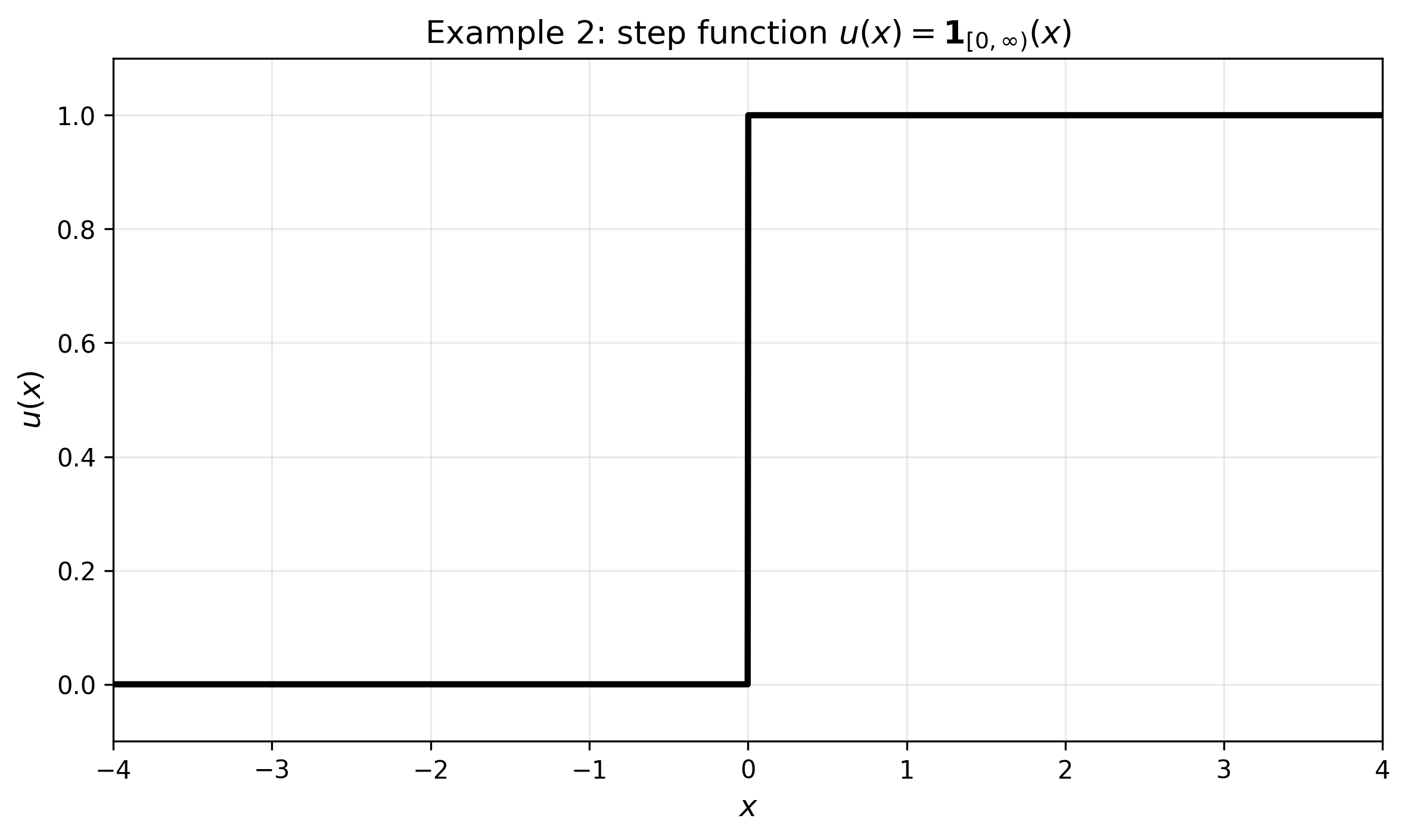}
				\includegraphics[width=.48\textwidth]{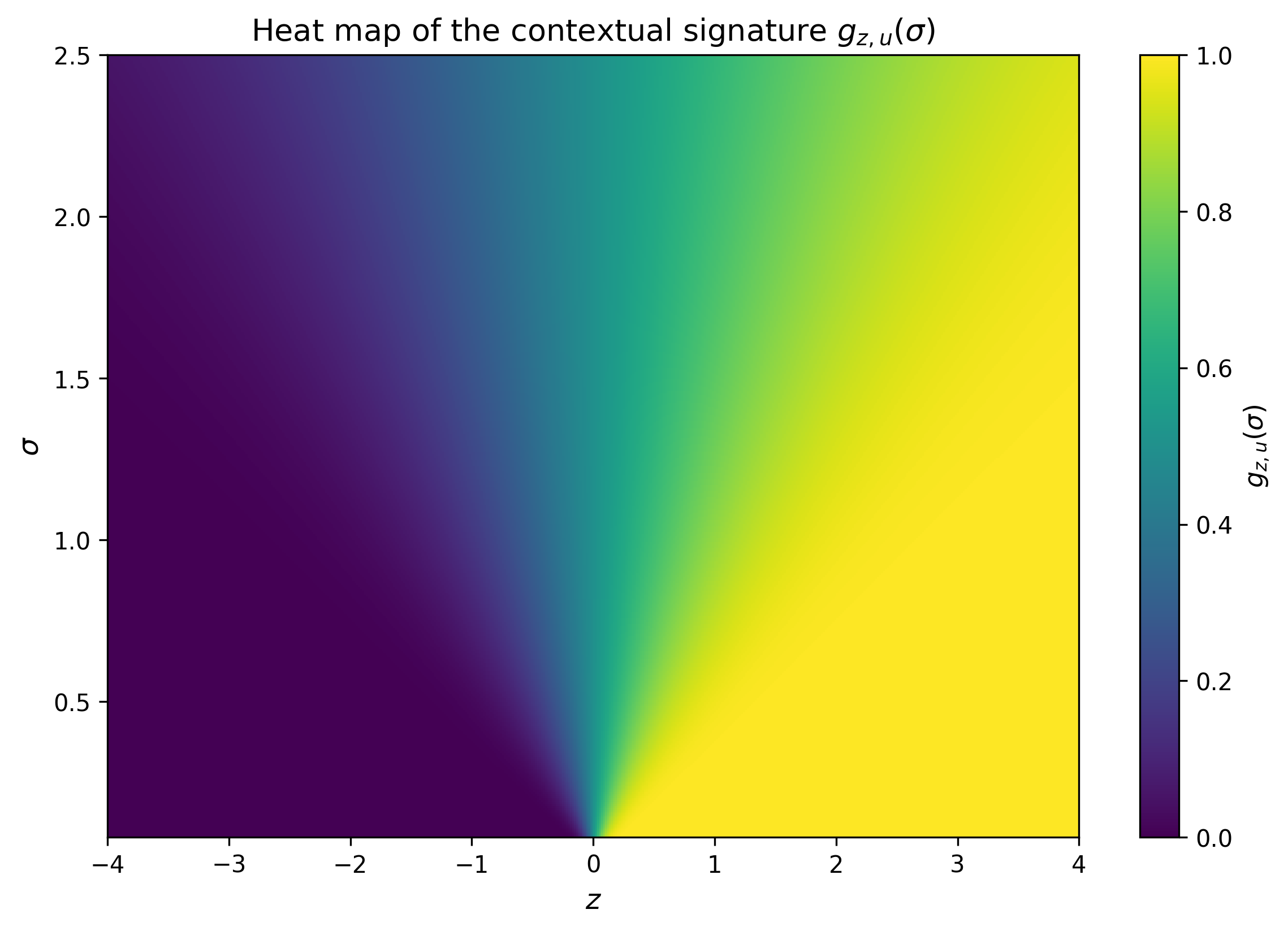}
				
				\vspace{0.3cm}
				
				\includegraphics[width=.48\textwidth]{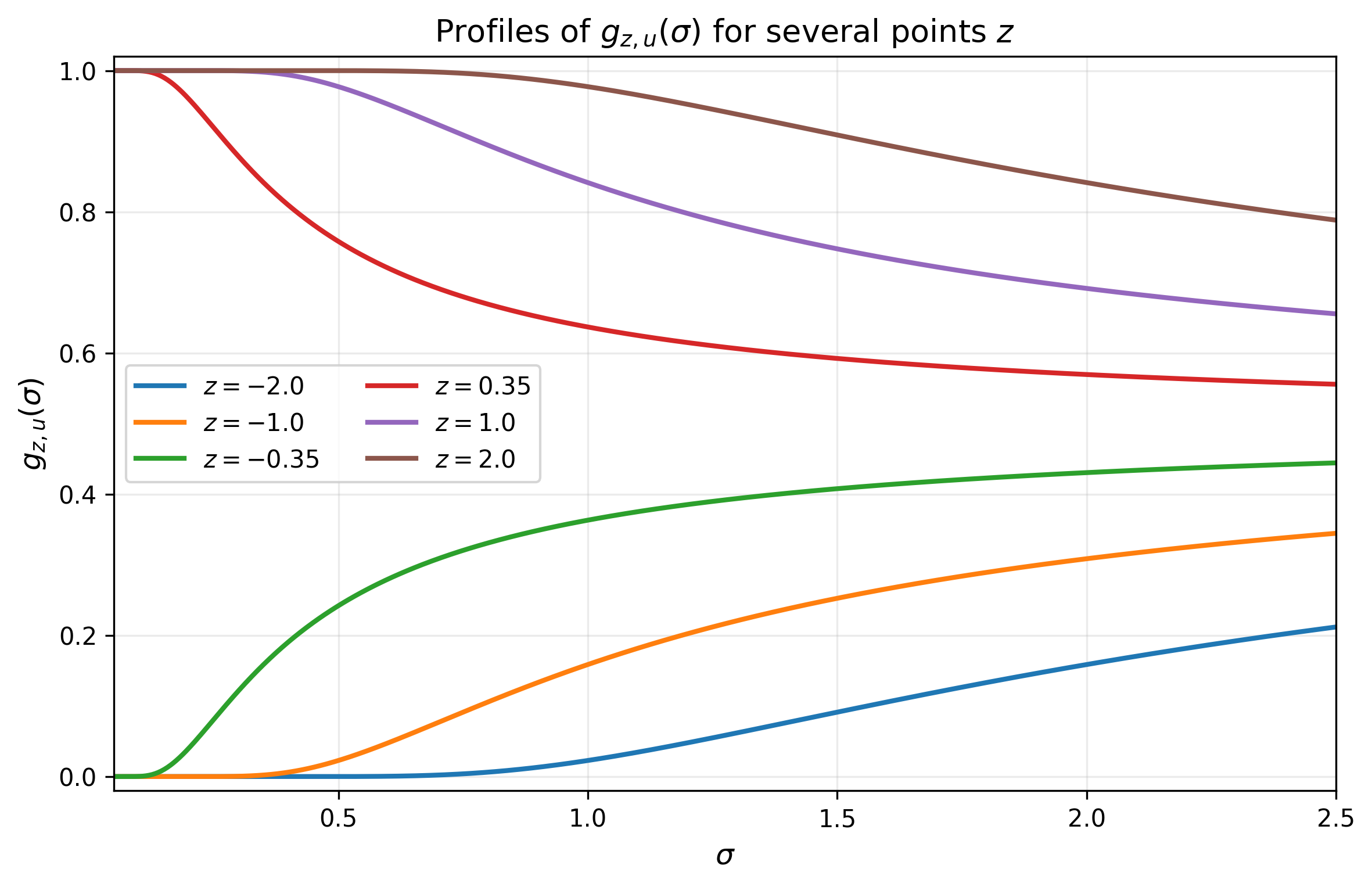}
				\includegraphics[width=.48\textwidth]{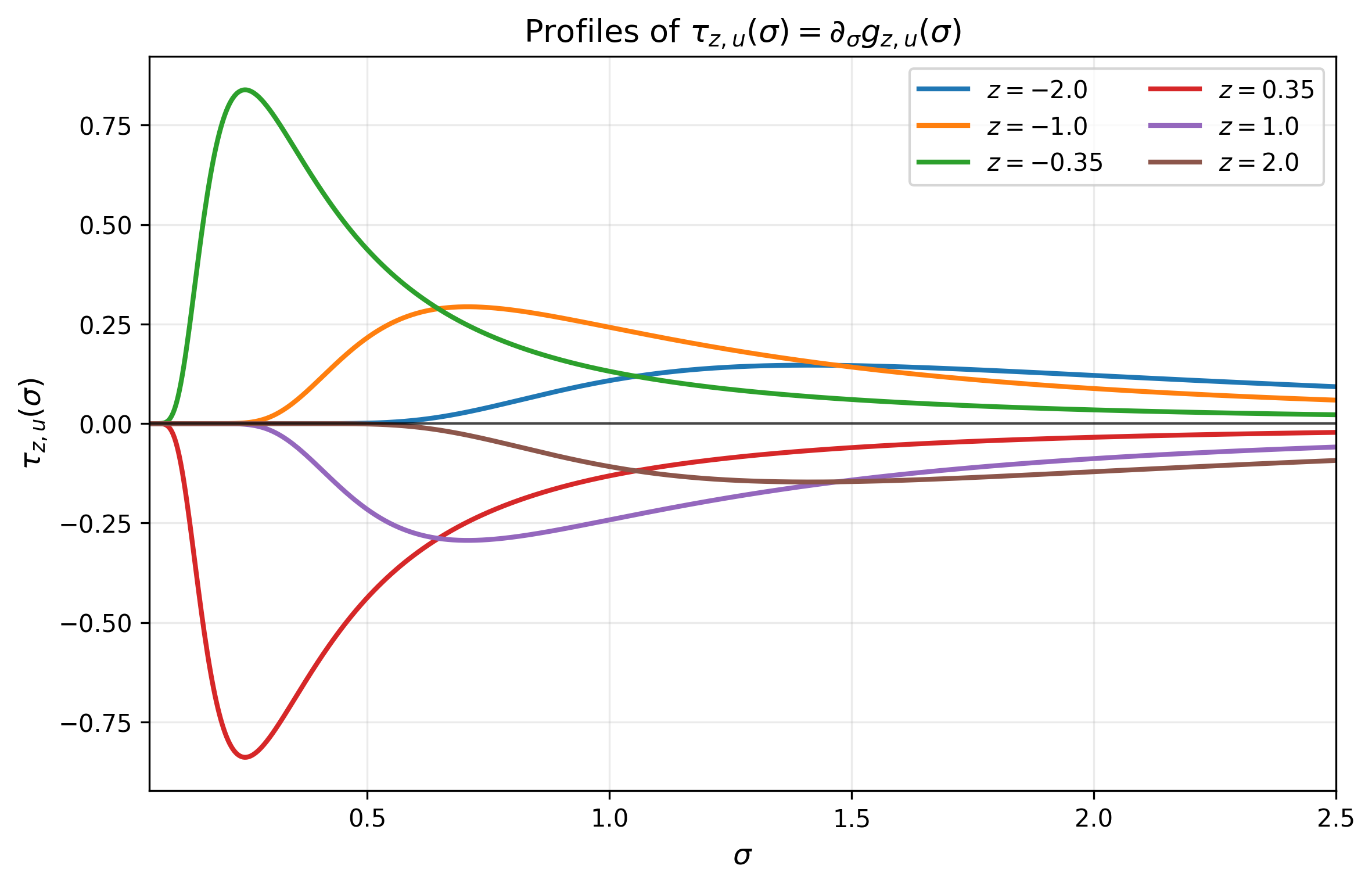}
				\caption{Illustrations of Example 2 for the ideal edge (Heaviside step function): original signal, multiscale contextual signature, and profiles of $g_{z,u}(\sigma)$ and $\tau_{z,u}(\sigma)$ around the discontinuity.}
			\end{figure}
			
			\medskip
			
			\noindent \textbf{Example 3: Sinusoidal texture.}
			Let $u(x) = \cos(\omega x)$ in one dimension. Its contextual signature is
			\begin{equation*}
				g_{z,u}(\sigma) = e^{-\omega^2\sigma^2/2} \cos(\omega z),
			\end{equation*}
			which decays exponentially with $\sigma$. The texture operator is given by
			\begin{equation*}
				\tau_{z,u}(\sigma) = -\omega^2 \sigma e^{-\omega^2\sigma^2/2} \cos(\omega z) = \omega^2 \sigma e^{-\omega^2\sigma^2/2} \cos(\omega z + \pi).
			\end{equation*}
			For $\sigma \ll 1/\omega$, $\tau$ is small because the texture has not yet been averaged out. It attains a maximum at $\sigma \approx 1/\omega$, the characteristic scale of the texture, and then decays to zero.
			
			For a sinusoid of frequency $\omega$, the response of the texture operator is maximal at the observation scale $\sigma = \frac{1}{w}$.
			
			\begin{figure}[htb]
				\centering
				\includegraphics[width=.48\textwidth]{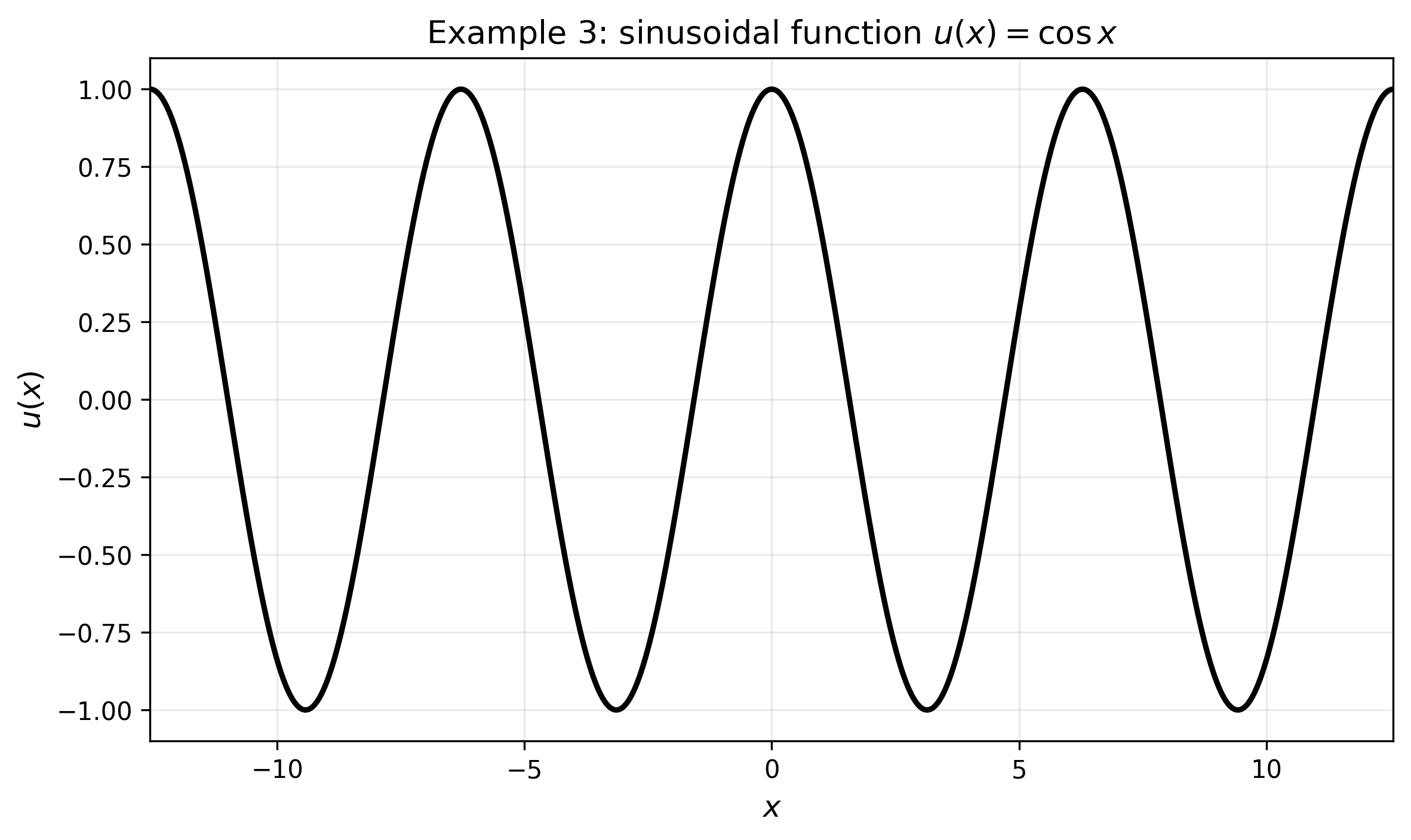}
				\includegraphics[width=.48\textwidth]{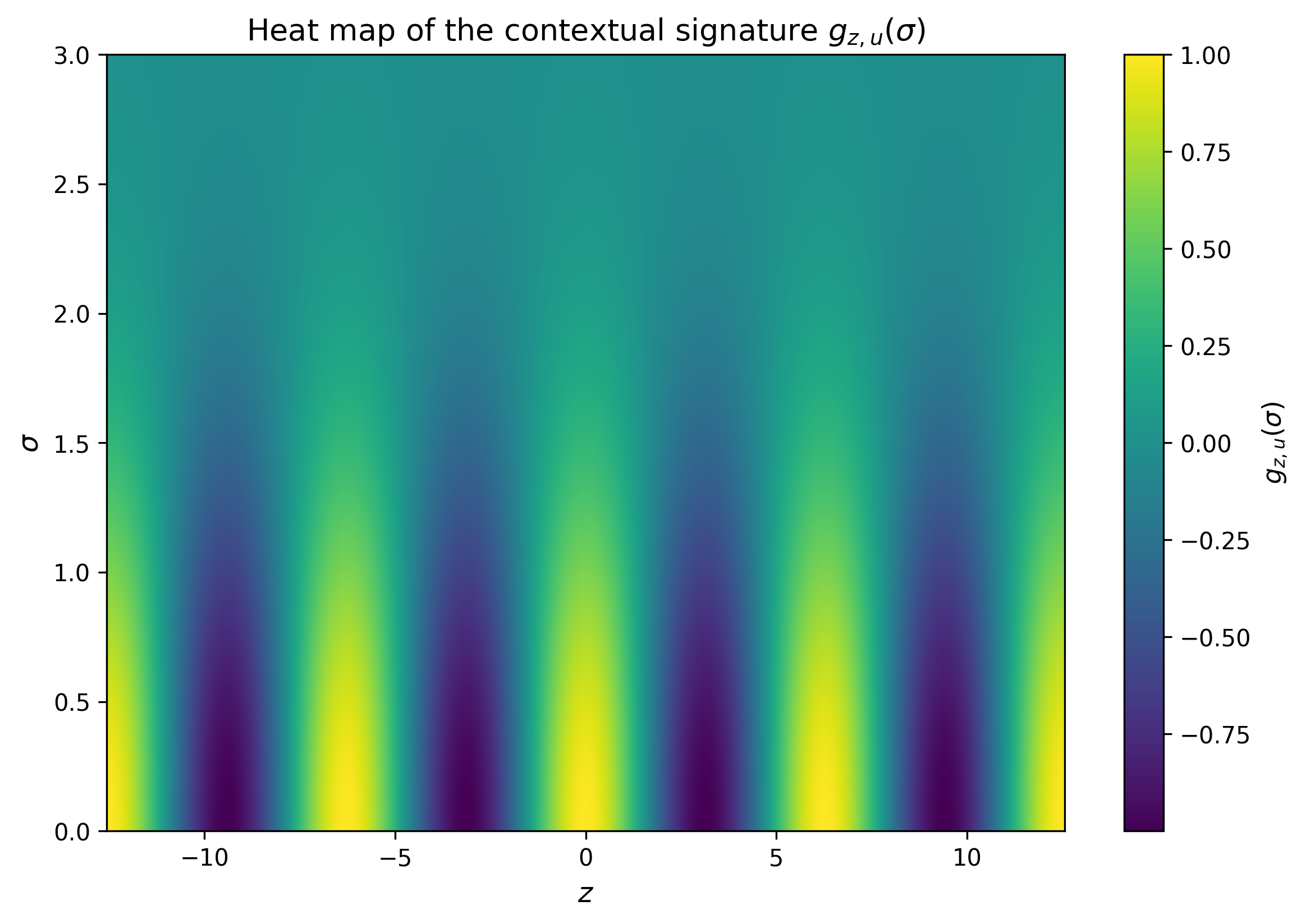}
				
				\vspace{0.3cm}
				
				\includegraphics[width=.48\textwidth]{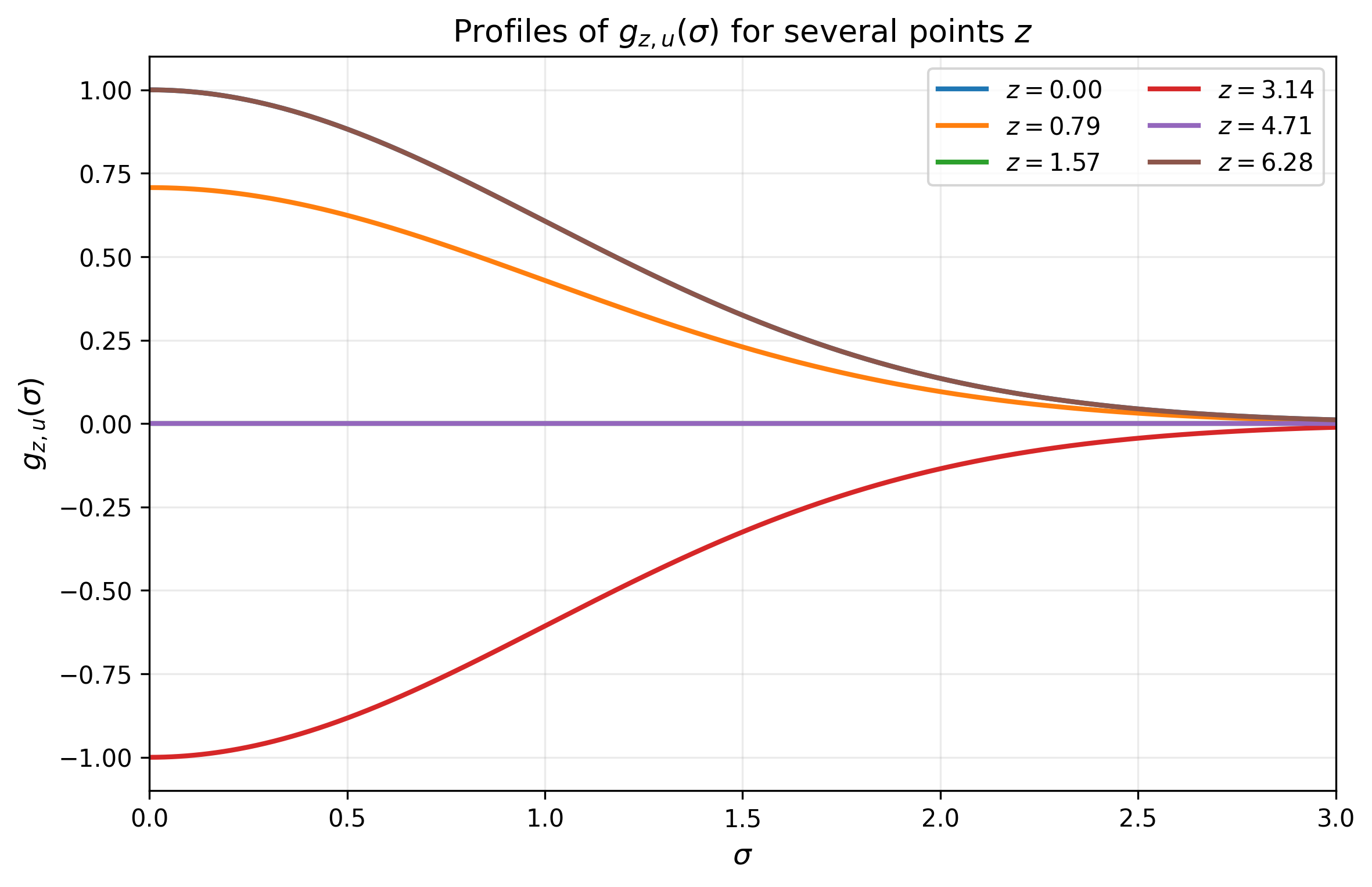}
				\includegraphics[width=.48\textwidth]{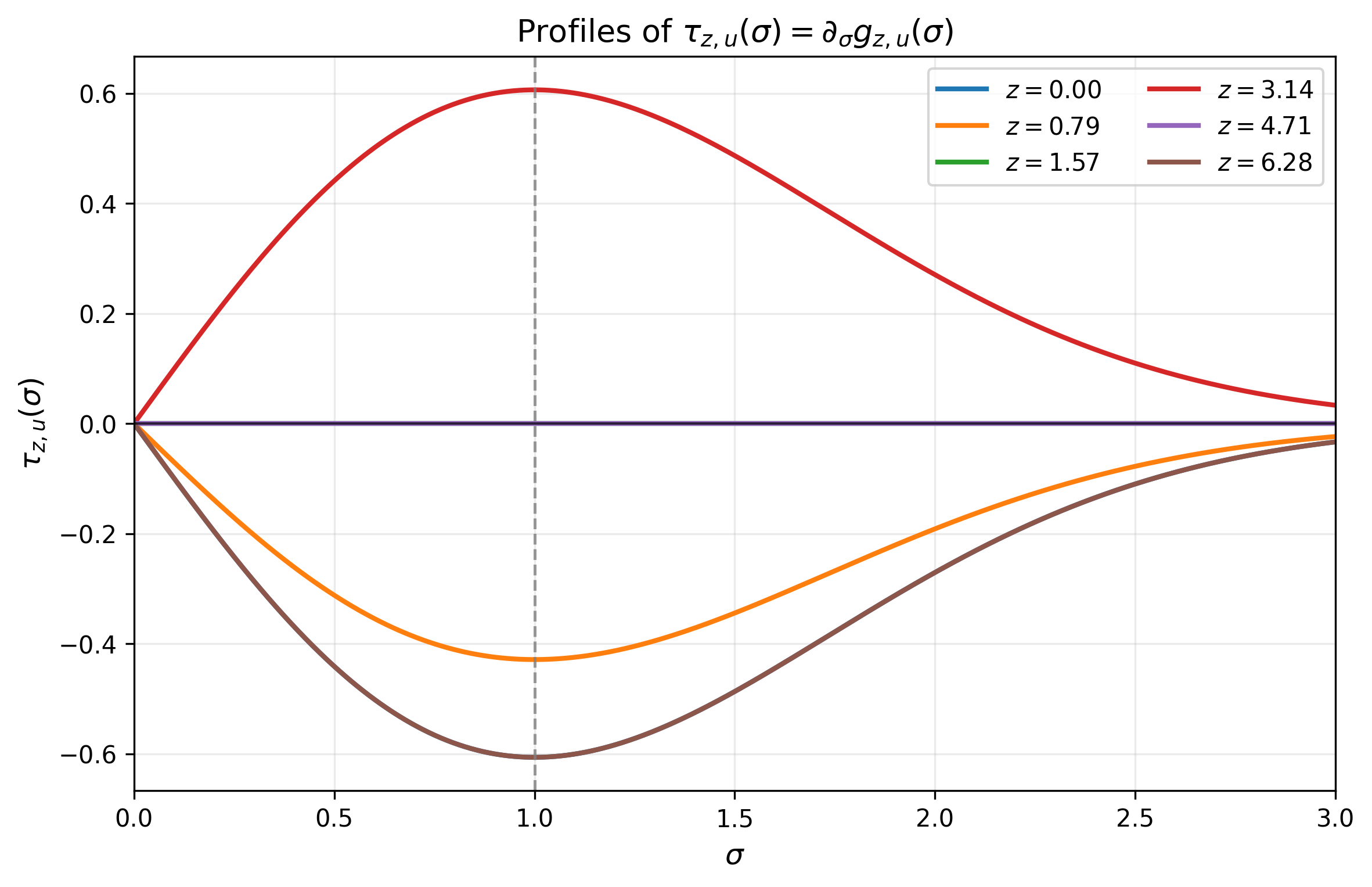}
				\caption{Illustrations of Example 3 for the sinusoidal texture $u(x)=\cos x$: original signal, heat map of $g_{z,u}(\sigma)$, and profiles of $g_{z,u}(\sigma)$ and $\tau_{z,u}(\sigma)$ showing resonance at the characteristic scale.}
			\end{figure}
			
			\medskip
			
			\noindent \textbf{Example 4: White noise.}
			If $u$ is Gaussian white noise, its Fourier transform is flat. Convolution with $\gamma_\sigma$ yields a process whose values converge, for large values of the parameter, to the mean value of the noise, and for which $\tau_{z,u}(\sigma)$ is a random variable with variance scaling like $\sigma^{-n-2}$. Formally, this scaling follows from the fact that $\tau_{z,u}(\sigma) = (u * \Phi_\sigma)(z)$ with $\Phi_\sigma = \sigma\Delta\gamma_\sigma$, whose $L^2$-norm is of order $\sigma^{-n/2-1}$; hence the variance of the convolution with white noise scales as $\|\Phi_\sigma\|_{L^2}^2 \sim \sigma^{-n-2}$. For small $\sigma$, the response is large and highly variable; for large $\sigma$, the noise is averaged out and the response decays. This observation anticipates the need for robustness to noise, which we will address in Section~\ref{sec:spectral}.
			
			\begin{figure}[htb]
				\centering
				\includegraphics[width=.48\textwidth]{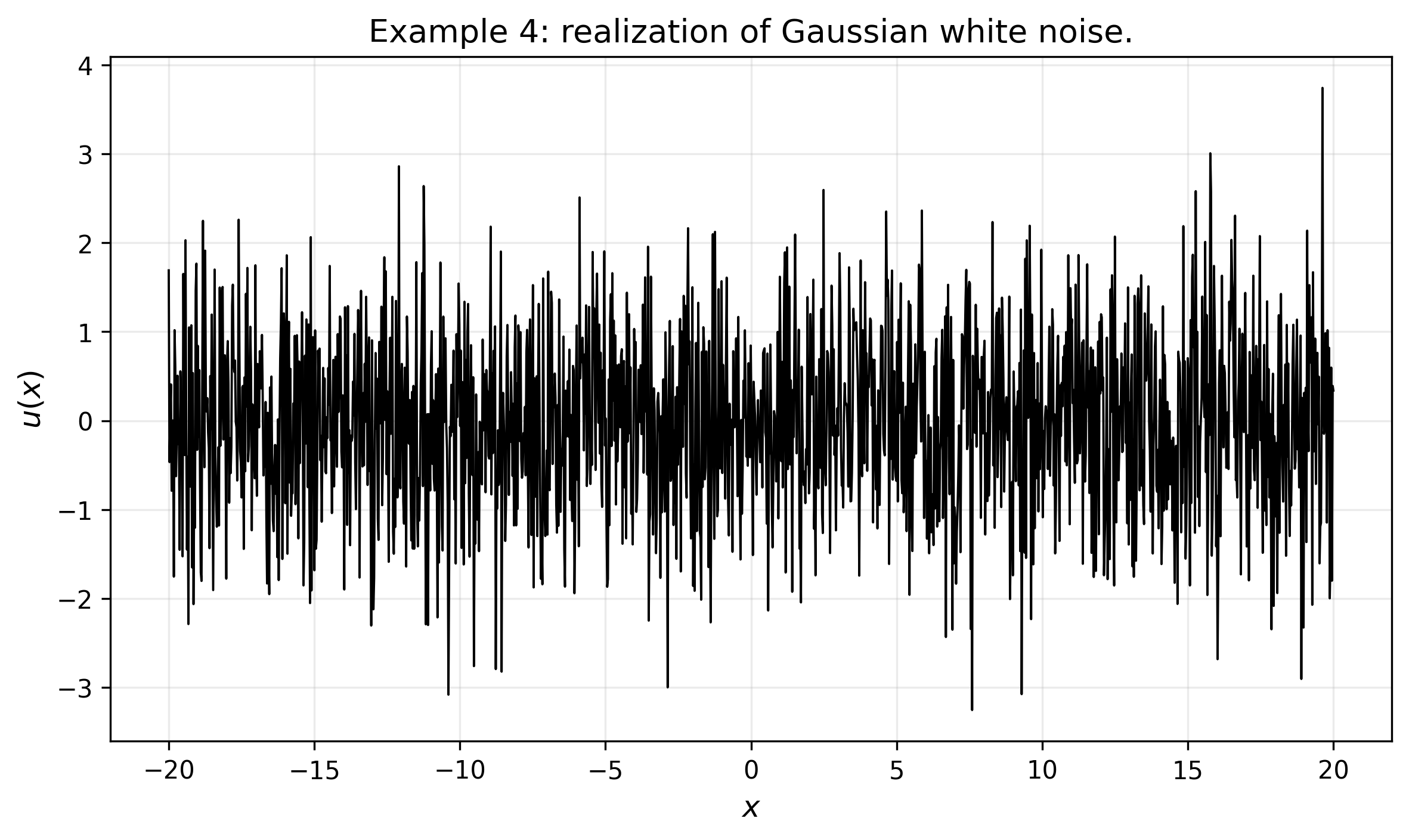}
				\includegraphics[width=.48\textwidth]{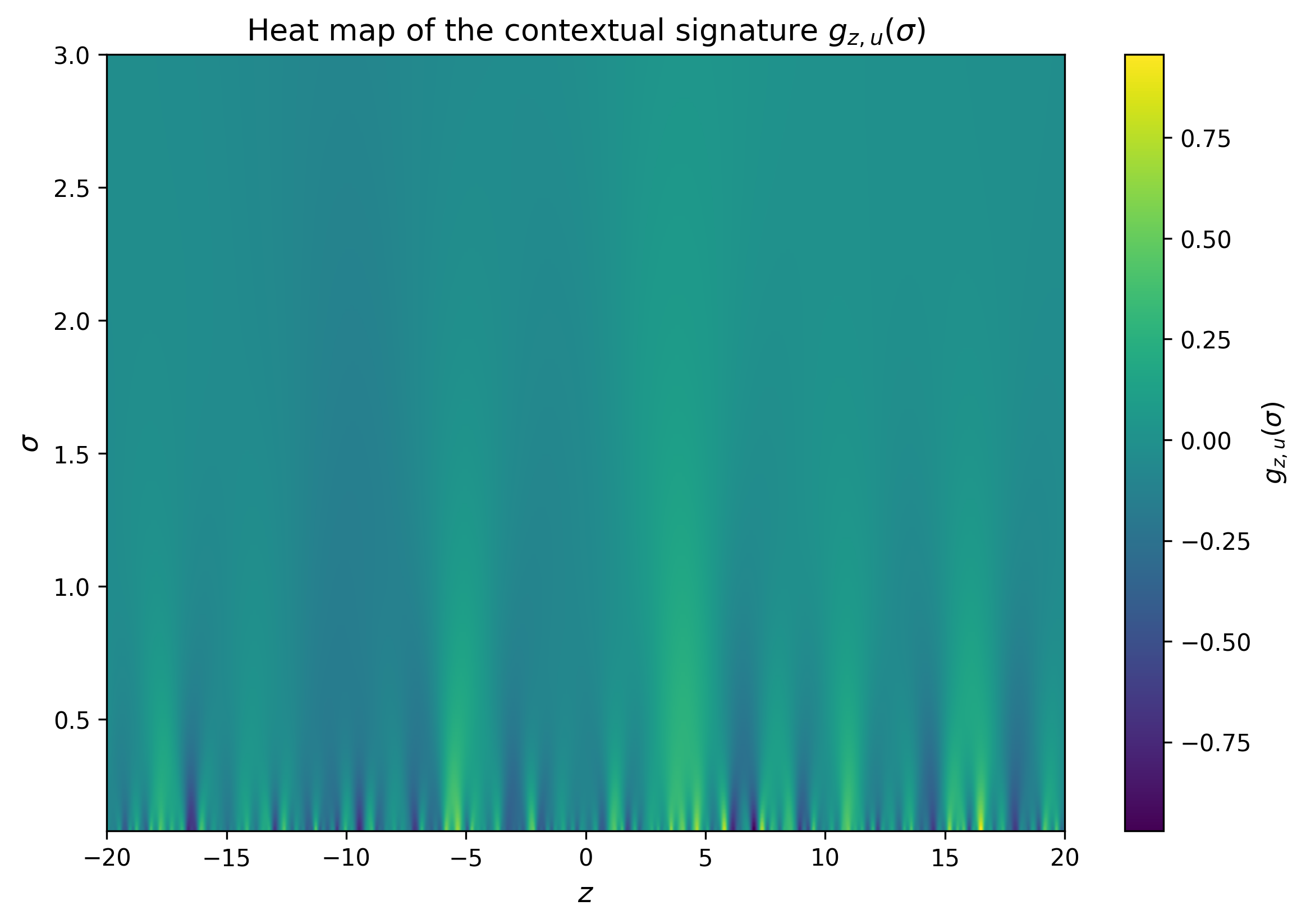}
				
				\vspace{0.3cm}
				
				\includegraphics[width=.48\textwidth]{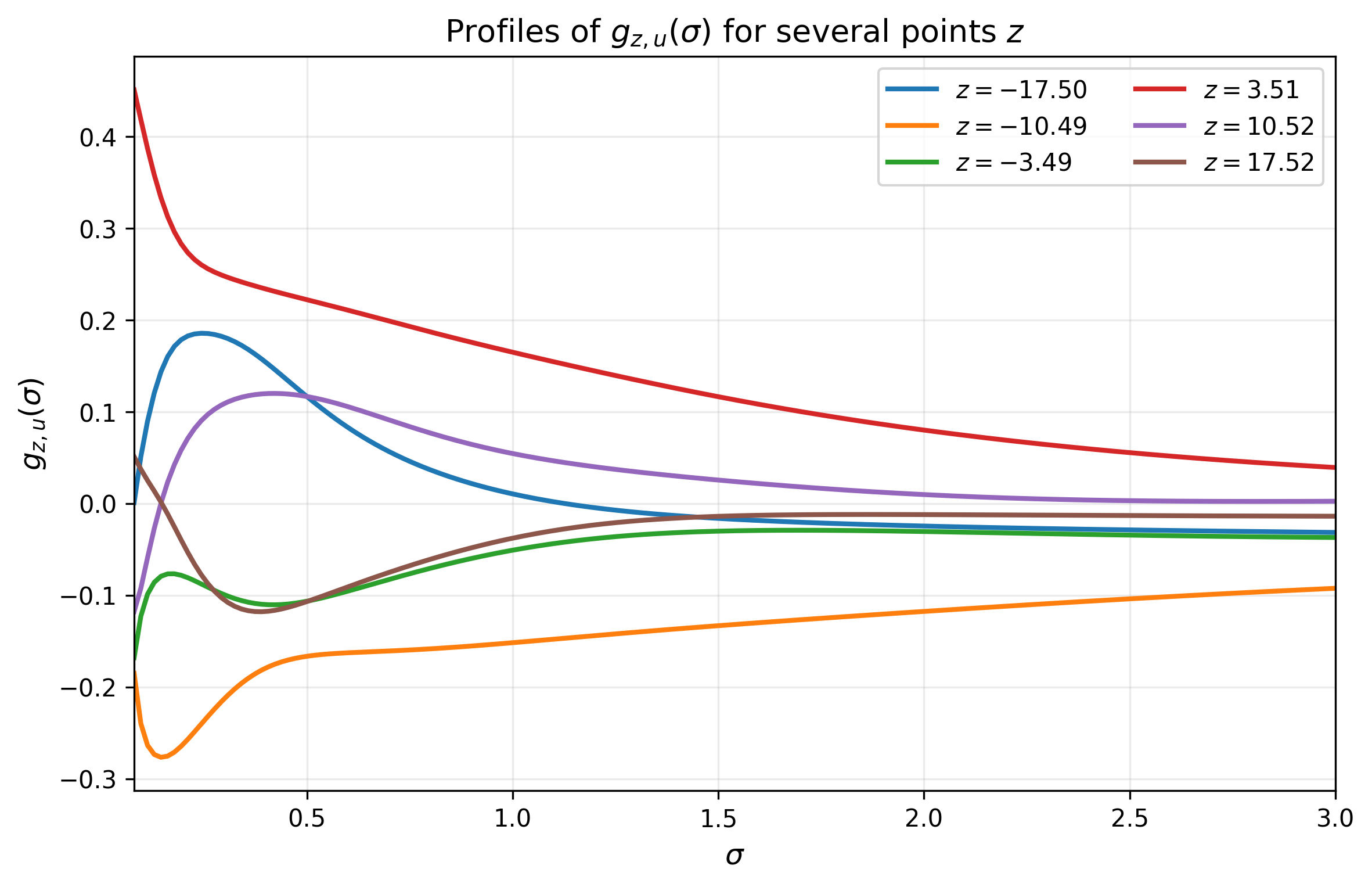}
				\includegraphics[width=.48\textwidth]{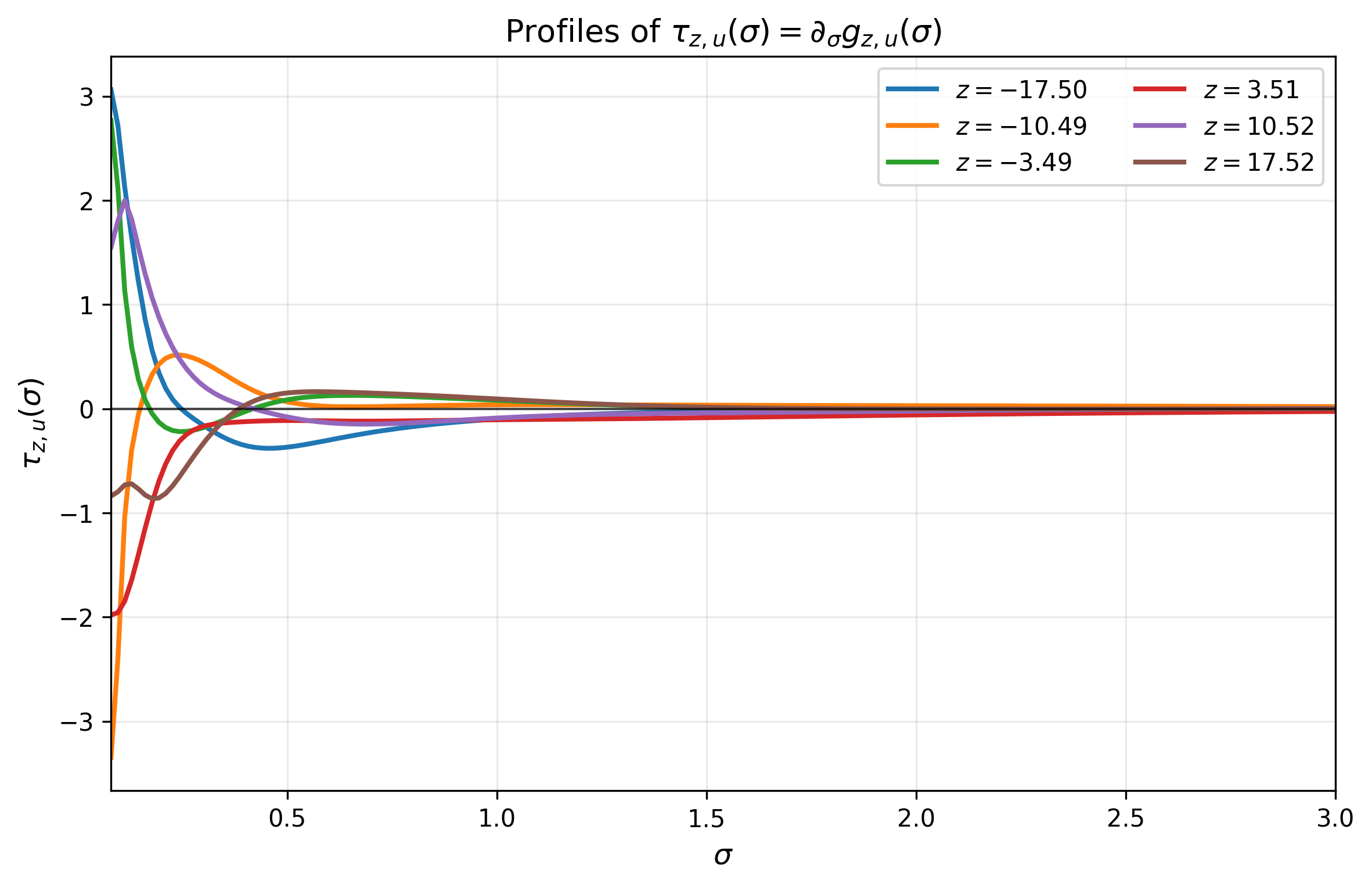}
				\caption{Illustrations of Example 4 for a realization of Gaussian white noise: original signal, Gaussian smoothings at different scales, and evolution of $g_{z,u}(\sigma)$ and $\tau_{z,u}(\sigma)$, showing the attenuation of noise as the scale increases.}
			\end{figure}
			
			\medskip
			
			\noindent \textbf{Example 5: Periodic texture with modulation.}
			Consider $u(x) = A(x)\cos(\omega x)$, where $A(x)$ is a slowly varying envelope. For scales $\sigma$ much smaller than the variation scale of $A$, $\tau$ detects the frequency $\omega$; for larger scales, the envelope modulates the amplitude of the response. This example shows how $\tau$ can simultaneously capture information about fine-scale texture (carrier frequency) and coarse-scale structure (modulation).
			
			\begin{figure}[htb]
				\centering
				\includegraphics[width=.48\textwidth]{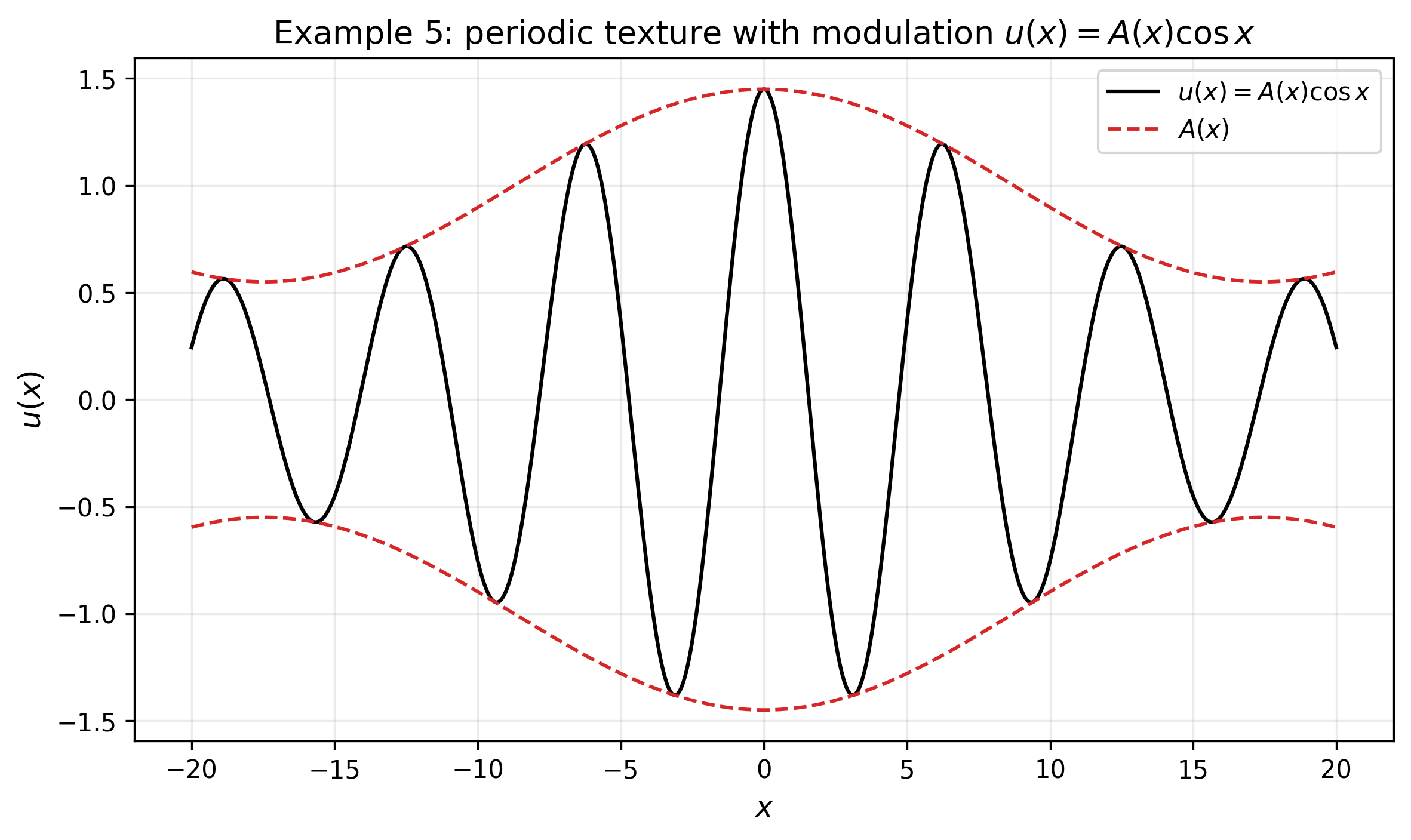}
				\includegraphics[width=.48\textwidth]{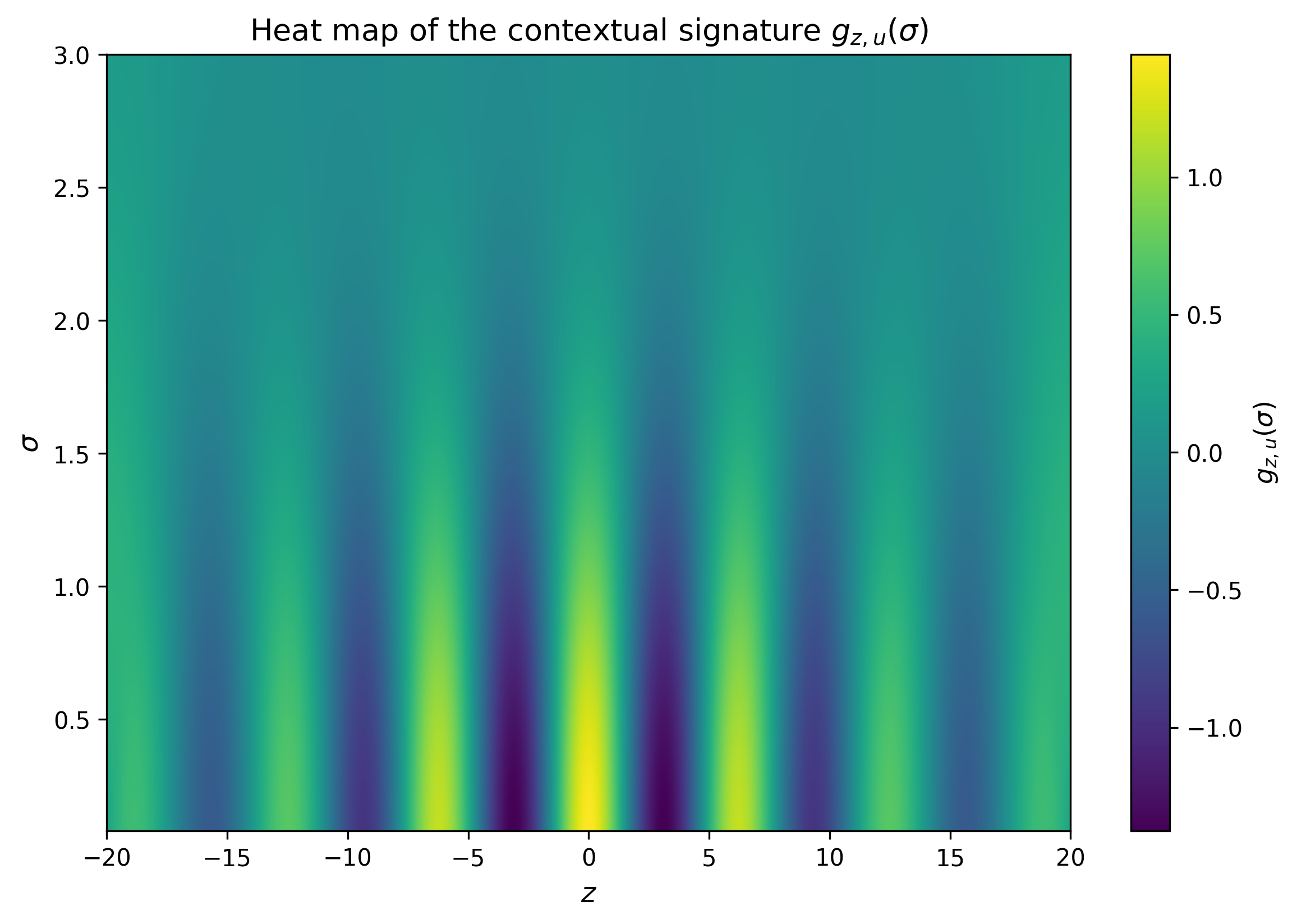}
				
				\vspace{0.3cm}
				
				\includegraphics[width=.48\textwidth]{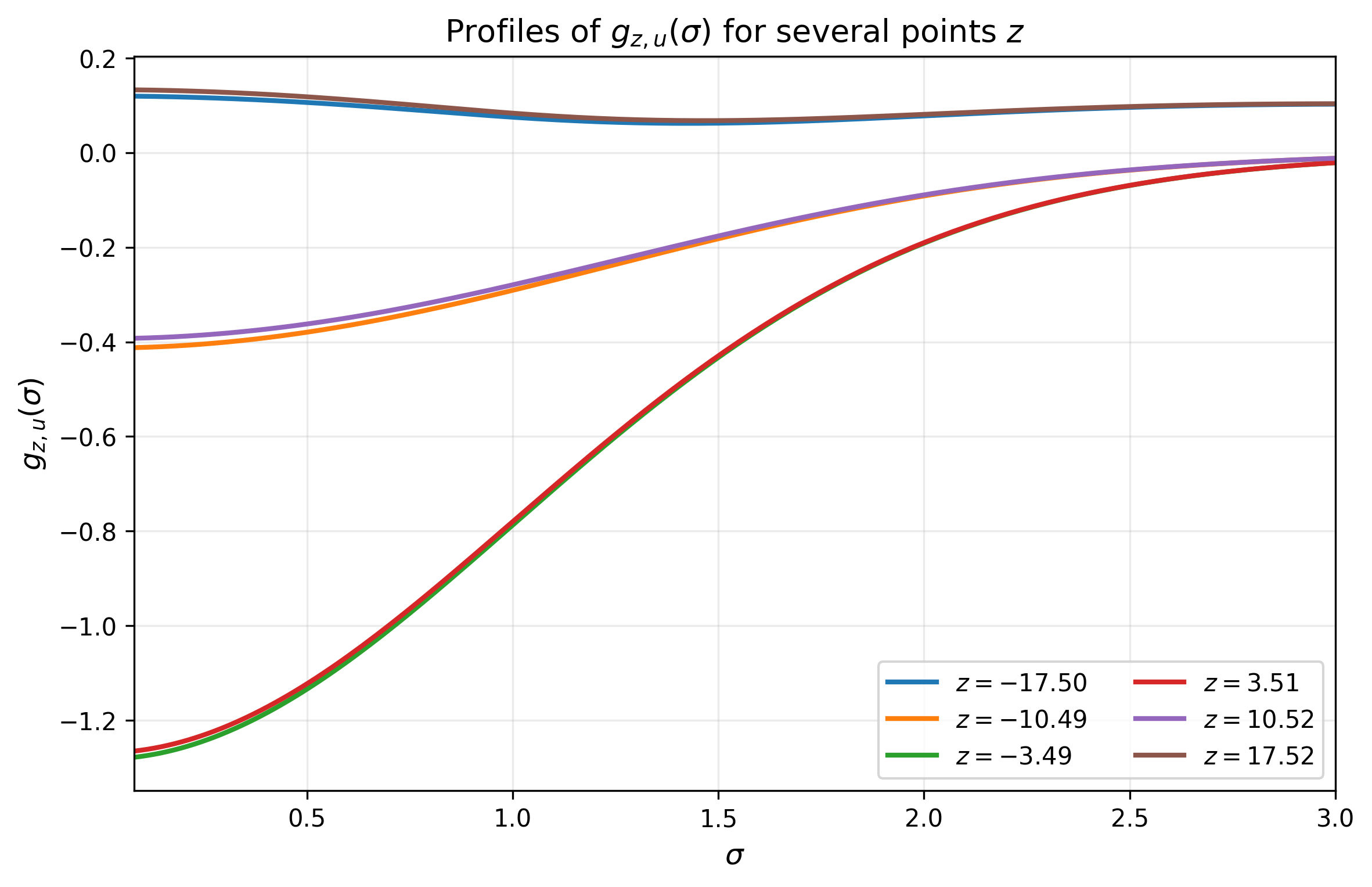}
				\includegraphics[width=.48\textwidth]{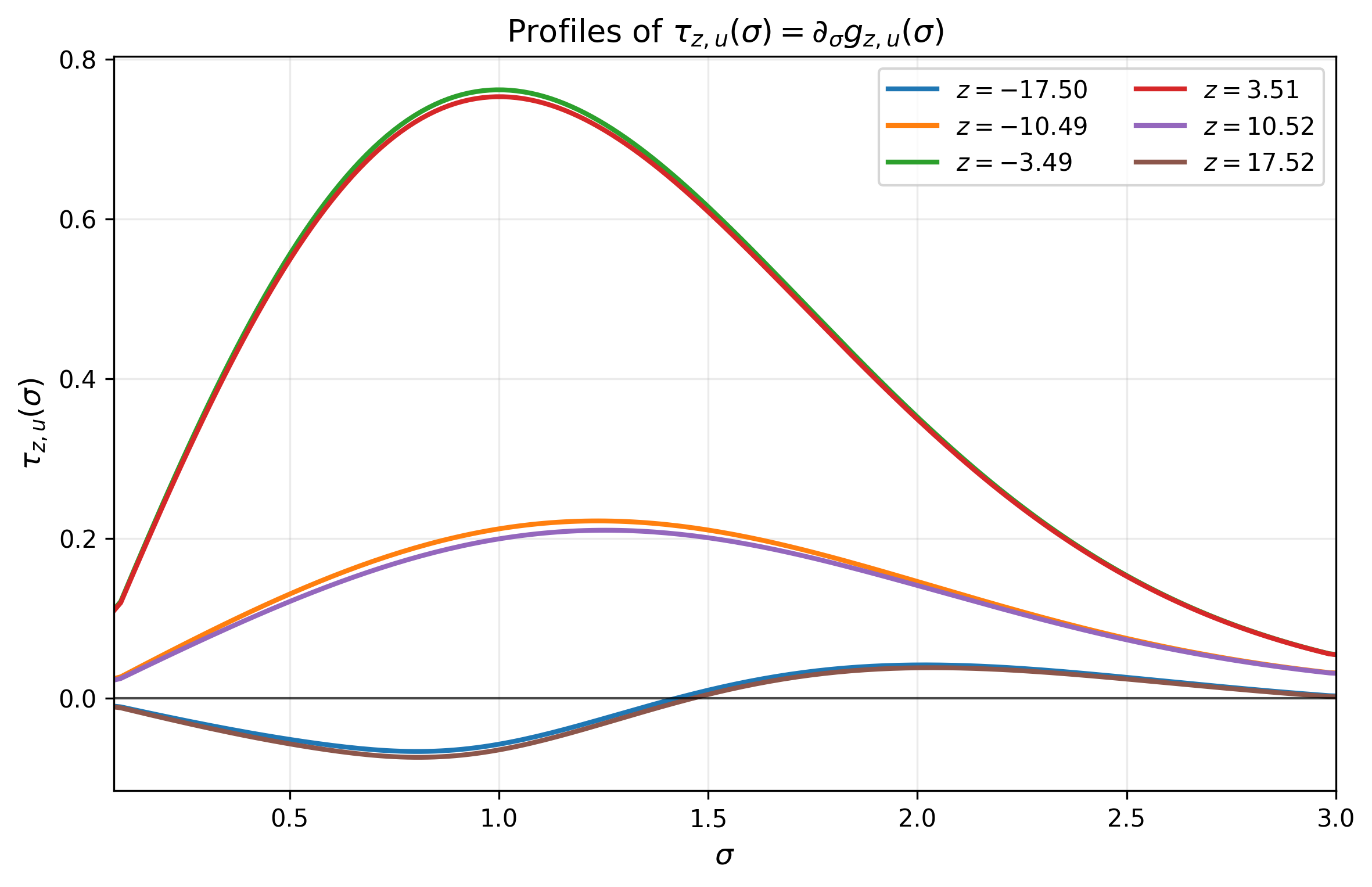}
				\caption{Illustrations of Example 5 for the modulated periodic texture $u(x)=A(x)\cos x$: original texture together with its envelope, heat map of $g_{z,u}(\sigma)$, and profiles of $g_{z,u}(\sigma)$ and $\tau_{z,u}(\sigma)$ reflecting the interaction between carrier frequency and slow modulation.}
			\end{figure}
			
			\medskip
			
			These examples illustrate the qualitative behavior of the pointwise multiscale texture operator in several canonical situations. The analytical results established in the following sections provide a rigorous mathematical explanation of the properties observed here, including boundedness, spectral localization, multiscale decomposition, and the Hilbertian functional framework associated with Gaussian scale evolution.

	\section{Definition and basic properties of the texture operator}\label{sec:basicproperties}
	
		This section introduces the pointwise multiscale texture operator and develops its basic analytical framework. Building upon the Gaussian scale-space representation, we define the multiscale contextual signature and its scale derivative, and establish the structural properties that will support the spectral and multiscale analysis developed in the following sections.
		
		\subsection{Definition of $g_{z,u}$ and $\tau_{z,u}$}
		
		We start from the classical Gaussian scale-space representation of a function. This construction naturally gives rise to a family of local averages parameterized by the scale variable. The multiscale contextual signature and the pointwise multiscale texture operator are obtained by studying the evolution of these averages as the scale varies.
		
		\begin{definition}
			Let $u \in L^p(\R^n)$, with $1\le p\le \infty$, and let $z \in \R^n$. The \emph{multiscale contextual signature} of $u$ at the point $z$ is defined by
			\begin{equation*}
				g_{z,u}(\sigma) = (u * \gamma_\sigma)(z) = \int_{\R^n}  u(y)\,\gamma_\sigma(z-y)\,dy.
			\end{equation*}
			The \emph{pointwise multiscale texture operator} is defined as the scale derivative of this signature:
			\begin{equation*}
				\tau_{z,u}(\sigma) = \frac{\partial}{\partial \sigma} g_{z,u}(\sigma).
			\end{equation*}
			For $u\in\mathcal S(\R^n)$, differentiation under the integral sign yields
			\begin{equation*}
				\tau_{z,u}(\sigma) = (u * \partial_\sigma \gamma_\sigma)(z).
			\end{equation*}
		\end{definition}
		
		For $u\in\mathcal S(\R^n)$, the convolution identity follows directly by differentiation under the integral sign. The convolution representation $u \mapsto u * \Phi_\sigma$ will later be extended by continuity to $L^p(\R^n)$ using its $L^p$-boundedness (see Proposition~\ref{prop:Lp-boundedness}). Although the pointwise multiscale texture operator was introduced in \cite{fernandez2}, it arises naturally from the Gaussian scale-space evolution introduced by Witkin and subsequently developed by Koenderink and Lindeberg \cite{witkin,koenderink,lindeberg}. The present work is devoted to the analytical study of this operator.
		
		From a qualitative standpoint, the function $g_{z,u}(\sigma)$ describes how the local structure around the point $z$ evolves as the observation scale increases. Small values of $\sigma$ capture fine-scale information, whereas larger values reflect progressively coarser structures. The texture operator $\tau_{z,u}(\sigma)$ measures the rate of variation of this local description with respect to scale. Regions where $|\tau_{z,u}(\sigma)|$ remains small are comparatively homogeneous across the corresponding range of scales, while large values indicate the presence of significant structural or textural changes.
		
		\subsection{Diffusion identity and kernel representation}
		
			The Gaussian scale-space representation provides a convenient analytical framework for the study of the texture operator. In particular, the operator admits an explicit diffusion interpretation, a convolution representation through a scale-dependent kernel, and a precise spectral characterization. The following results summarize the basic analytical properties of this construction.
			
			\begin{proposition}[Diffusion identity in scale-space]\label{prop:diffusion}
				For every $u \in L^p(\R^n)$, $1\le p \le \infty$, and every $\sigma>0$, one has
				\begin{equation*}
					\frac{\partial}{\partial \sigma} g_{z,u}(\sigma) = \sigma \Delta g_{z,u}(\sigma),
				\end{equation*}
				where $\Delta$ denotes the spatial Laplacian. Consequently,
				\begin{equation*}
					\tau_{z,u}(\sigma) = \sigma \Delta (u * \gamma_\sigma)(z).
				\end{equation*}
			\end{proposition}
			
			As a direct consequence, the texture operator admits the convolution representation
			\begin{equation*}
				\tau_{z,u}(\sigma) = (u * \Phi_\sigma)(z),
			\end{equation*}
			where
			\begin{equation*}
				\Phi_\sigma(x) = \sigma \Delta \gamma_\sigma(x) = \frac{\sigma}{(2\pi\sigma^2)^{n/2}} \left(\frac{|x|^2}{\sigma^2} - n\right) \exp\!\left(-\frac{|x|^2}{2\sigma^2}\right).
			\end{equation*}
			
			\begin{proposition}\label{prop:TF_nucleo}
				The Fourier transform of $\Phi_\sigma$ is given by
				\begin{equation*}
					\widehat{\Phi_\sigma}(\xi) = -\sigma |\xi|^2 \exp\!\left(-\frac{\sigma^2|\xi|^2}{2}\right).
				\end{equation*}
			\end{proposition}
			
			\begin{proof}
				By definition, $\Phi_\sigma = \sigma \Delta \gamma_\sigma$. Taking Fourier transforms and using the convention of Section~\ref{sec:preliminares}, we obtain
				\begin{equation*}
					\widehat{\Phi_\sigma}(\xi) = \sigma\,\widehat{\Delta \gamma_\sigma}(\xi) = \sigma\,(-|\xi|^2)\,\widehat{\gamma_\sigma}(\xi).
				\end{equation*}
				Since $\widehat{\gamma_\sigma}(\xi) = \exp(-\sigma^2|\xi|^2/2)$, this yields
				\begin{equation*}
					\widehat{\Phi_\sigma}(\xi) = -\sigma |\xi|^2 \exp\!\left(-\frac{\sigma^2|\xi|^2}{2}\right),
				\end{equation*}
				which proves the claim.
			\end{proof}
			
			The explicit form of $\widehat{\Phi_\sigma}$ immediately reveals the scale-dependent, frequency-selective nature of the texture operator.
			
			\begin{theorem}[Band-pass properties]\label{thm:bandpass}
				The symbol $\widehat{\Phi_\sigma}$ satisfies:
				\begin{enumerate}
					\item $\widehat{\Phi_\sigma}(0)=0$.
					\item $\widehat{\Phi_\sigma}(\xi)<0$ for every $\xi\neq 0$, and the unique zero is at $\xi=0$.
					\item The unique global minimum is attained at
					\begin{equation*}
						|\xi|_{\min} = \frac{\sqrt{2}}{\sigma}.
					\end{equation*}
					\item
					\begin{equation*}
						\lim_{|\xi|\to\infty} \widehat{\Phi_\sigma}(\xi) = 0.
					\end{equation*}
				\end{enumerate}
				Consequently, $\Phi_\sigma$ acts as a Gaussian band-pass filter whose response is concentrated around frequencies of order $1/\sigma$, with peak response at $\sqrt{2}/\sigma$. The symbol is nonpositive and vanishes only at the origin, reflecting the fact that $\Phi_\sigma$ is a zero-mean, scale-dependent Laplacian of the Gaussian.
			\end{theorem}
			
			\begin{proof}
				Let $r = |\xi| \ge 0$. From Proposition~\ref{prop:TF_nucleo}, the symbol is given by
				\begin{equation*}
					\widehat{\Phi_\sigma}(\xi) = -\sigma r^2 \exp\!\left(-\frac{\sigma^2 r^2}{2}\right).
				\end{equation*}
				We analyze each property separately.
				
				\begin{enumerate}
					\item At $\xi = 0$, we have $r = 0$, hence
					\begin{equation*}
						\widehat{\Phi_\sigma}(0) = -\sigma \cdot 0^2 \cdot e^{0} = 0.
					\end{equation*}
					
					\item For $\xi \neq 0$, we have $r > 0$. Since $\sigma > 0$, $r^2 > 0$ and $\exp(-\sigma^2 r^2/2) > 0$, it follows that
					\begin{equation*}
						\widehat{\Phi_\sigma}(\xi) = -\sigma r^2 \exp\!\left(-\frac{\sigma^2 r^2}{2}\right) < 0.
					\end{equation*}
					Moreover, the equation $\widehat{\Phi_\sigma}(\xi) = 0$ is therefore equivalent to $r=0$, that is, $\xi = 0$.
					
					\item To locate the critical points, we differentiate with respect to $r$:
					\begin{equation*}
						\frac{d}{dr} \widehat{\Phi_\sigma}(r) = \frac{d}{dr}\left(-\sigma r^2 e^{-\sigma^2 r^2/2}\right).
					\end{equation*}
					Using the product rule,
					\begin{equation*}
						\frac{d}{dr} \widehat{\Phi_\sigma}(r) = -2\sigma r e^{-\sigma^2 r^2/2} + \sigma^3 r^3 e^{-\sigma^2 r^2/2}.
					\end{equation*}
					Factoring the common term $\sigma r e^{-\sigma^2 r^2/2}$, we obtain
					\begin{equation*}
						\frac{d}{dr} \widehat{\Phi_\sigma}(r) = -\sigma r e^{-\sigma^2 r^2/2} \left(2 - \sigma^2 r^2\right).
					\end{equation*}
					Since $\sigma > 0$ and $r \ge 0$, the critical points satisfy $2 - \sigma^2 r^2 = 0$, so
					\begin{equation*}
						r = \frac{\sqrt{2}}{\sigma}.
					\end{equation*}
					For $0 < r < \sqrt{2}/\sigma$, one has $2 - \sigma^2 r^2 > 0$, hence $\widehat{\Phi_\sigma}'(r) < 0$, while for $r > \sqrt{2}/\sigma$ one has $2 - \sigma^2 r^2 < 0$, so $\widehat{\Phi_\sigma}'(r) > 0$. Thus the function decreases on $(0, \sqrt{2}/\sigma)$ and increases on $(\sqrt{2}/\sigma, \infty)$, and the point $r = \sqrt{2}/\sigma$ is a global minimum. This proves (3).
					
					\item Finally,
					\begin{equation*}
						\lim_{|\xi| \to \infty} \widehat{\Phi_\sigma}(\xi) = -\lim_{r \to \infty} \sigma r^2 e^{-\sigma^2 r^2/2} = -\sigma \lim_{r \to \infty} \frac{r^2}{e^{\sigma^2 r^2/2}} = 0,
					\end{equation*}
					since the exponential factor dominates any polynomial. This proves (4).
				\end{enumerate}
				
			\end{proof}
			
			\begin{figure}[htb]
				\centering
				\includegraphics[width=0.72\textwidth]{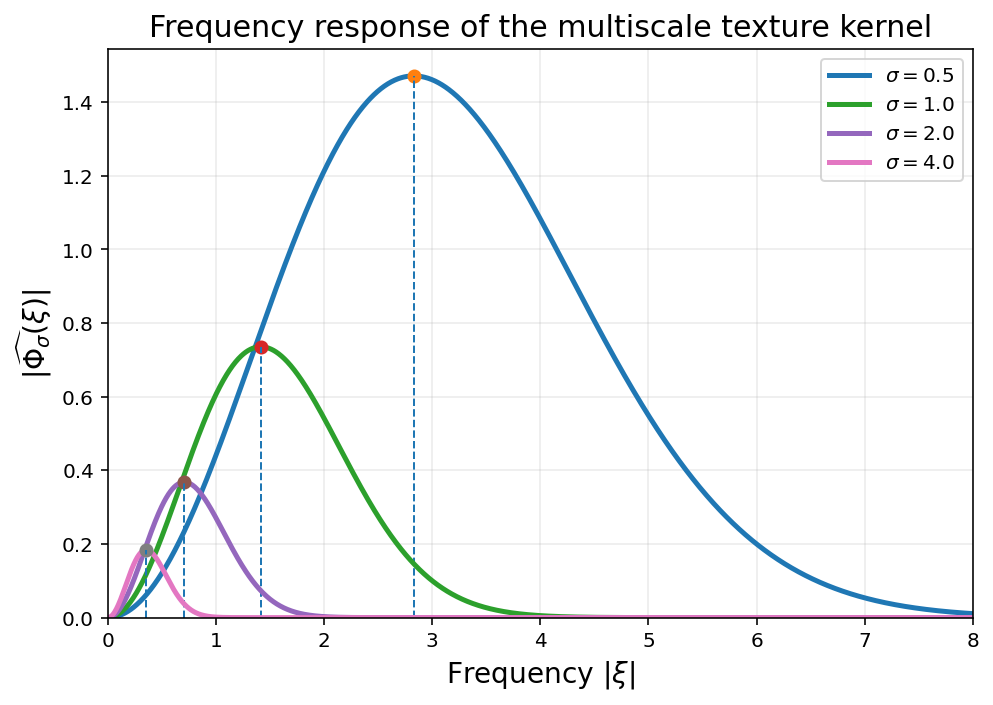}
				\caption{Frequency response of the multiscale texture kernel for several values of the scale parameter $\sigma$. The magnitude of the symbol, $|\widehat{\Phi_\sigma}(\xi)|=\sigma |\xi|^2\exp\!\left(-\frac{\sigma^2|\xi|^2}{2}\right)$, acts as a Gaussian band-pass filter. The dashed vertical lines indicate the characteristic frequency $|\xi|=\sqrt{2}/\sigma$, where the response attains its maximum. Increasing the observation scale shifts the pass band towards lower frequencies while preserving the same qualitative profile.} \label{fig:bandpass}
			\end{figure}
			
			Consequently the operator suppreses both very low and very high frequencies. These spectral properties provide the analytical foundation for the scale-selective behaviour of the texture operator. Their consequences can already be observed in several elementary examples, which we discuss next.
		
		\subsection{Boundedness and smoothing}
			
			In this subsection we show that the texture operator acts boundedly on $L^p$ spaces and, in the Hilbertian setting, exhibits strong regularizing properties. These results follow from classical estimates for the Gaussian kernel and its derivatives and provide the analytical foundation for the multiscale texture spaces introduced later in the paper.
			
			\begin{proposition}\label{prop:Lp-boundedness}
				For every $\sigma>0$ and every $1\le p\le\infty$, the operator $u\mapsto\tau_{\cdot,u}(\sigma)$ defines a bounded linear map on $L^p(\R^n)$. More precisely,
				\begin{equation*}
					\|\tau_{\cdot,u}(\sigma)\|_{L^p}\le C(\sigma)\,\|u\|_{L^p}.
				\end{equation*}
			\end{proposition}
			
			\begin{proof}
				By Proposition~\ref{prop:diffusion}, we have $\tau_{\cdot,u}(\sigma)=u*\Phi_\sigma$. Since $\Phi_\sigma\in L^1(\R^n)$ for every $\sigma>0$, Young's inequality yields
				\begin{equation*}
					\|\tau_{\cdot,u}(\sigma)\|_{L^p}=\|u*\Phi_{\sigma}\|_{L^p}\le\|\Phi_{\sigma}\|_{L^1}\,\|u\|_{L^p}.
				\end{equation*}
				The result follows by taking $C(\sigma)=\|\Phi_\sigma\|_{L^1}$.
			\end{proof}
			
			The previous proposition shows that, for each fixed scale $\sigma$, the pointwise multiscale texture operator defines a continuous linear operator on $L^p(\R^n)$. In the particular case $p=2$, the Gaussian factor appearing in the symbol of $\Phi_\sigma$ yields a much stronger property: the operator is infinitely smoothing.
			
			\begin{proposition}[Infinite-order smoothing]\label{prop:infinite-smoothing}
				Let $\sigma>0$ and $k\ge0$. Then the map
				\begin{equation*}
					u\mapsto\tau_{\cdot,u}(\sigma)
				\end{equation*}
				extends to a bounded linear operator
				\begin{equation*}
					L^2(\R^n)\longrightarrow H^k(\R^n).
				\end{equation*}
				Moreover, there exists a constant $C_k(\sigma)>0$ such that
				\begin{equation*}
					\|\tau_{\cdot,u}(\sigma)\|_{H^k}\le C_k(\sigma)\,\|u\|_{L^2},
				\end{equation*}
				for every $u\in L^2(\R^n)$.
			\end{proposition}
			
			\begin{proof}
				By Proposition~\ref{prop:TF_nucleo},
				\begin{equation*}
					\widehat{\tau_{\cdot,u}(\sigma)}(\xi)=\widehat{\Phi}_{\sigma}(\xi)\,\widehat{u}(\xi),
				\end{equation*}
				where
				\begin{equation*}
					\widehat{\Phi}_{\sigma}(\xi)=-\sigma|\xi|^2 e^{-\sigma^2|\xi|^2/2}.
				\end{equation*}
				Therefore,
				\begin{equation*}
					\|\tau_{\cdot,u}(\sigma)\|_{H^k}^2=\int_{\R^n}(1+|\xi|^2)^k\,|\widehat{\Phi_\sigma}(\xi)|^2\,|\widehat{u}(\xi)|^2\,d\xi.
				\end{equation*}
				Since the multiplier
				\begin{equation*}
					(1+|\xi|^2)^k|\widehat{\Phi_\sigma}(\xi)|^2 = \sigma^2 (1+|\xi|^2)^k |\xi|^4 e^{-\sigma^2|\xi|^2}
				\end{equation*}
				is smooth and decays exponentially as $|\xi|\to\infty$, it is bounded on $\R^n$. Hence,
				\begin{equation*}
					\|\tau_{\cdot,u}(\sigma)\|_{H^k}^2\le\sup_{\xi\in\R^n}\bigl[(1+|\xi|^2)^k|\widehat{\Phi_\sigma}(\xi)|^2\bigr]\int_{\R^n}|\widehat{u}(\xi)|^2\,d\xi.
				\end{equation*}
				Defining
				\begin{equation*}
					C_k(\sigma)^2=\sup_{\xi\in\R^n}(1+|\xi|^2)^k|\widehat{\Phi_\sigma}(\xi)|^2<\infty.
				\end{equation*}
				Applying Plancherel's theorem, we obtain
				\begin{equation*}
					\|\tau_{\cdot,u}(\sigma)\|_{H^k}\le C_k(\sigma)\,\|\widehat{u}\|_{L^2}=C_k(\sigma)\,\|u\|_{L^2},
				\end{equation*}
				which concludes the proof.
			\end{proof}
			
			The infinite-order smoothing property implies that, although the operator is designed to detect textural variations, its output remains highly regular at every fixed scale. This regularity will play a central role in the multiscale constructions developed in the following sections.
			
	\section{Stability and invariance of the texture operator} \label{sec:spectral}

		The results of the previous section establish the basic analytical properties of the texture operator. We now adopt a spectral viewpoint to study its stability and invariance properties. The Fourier representation of the kernel provides a natural framework for quantifying the effect of perturbations, understanding the dependence on the scale parameter, and motivating the multiscale constructions developed in the subsequent sections.
		
		\subsection{Band-pass properties}
		
			The explicit expression of the Fourier symbol
			\begin{equation*}
				\widehat{\Phi_\sigma}(\xi) = -\sigma |\xi|^2 \exp\!\left(-\frac{\sigma^2|\xi|^2}{2}\right)
			\end{equation*}
			reveals the frequency-selective nature of the texture operator. As established in Theorem~\ref{thm:bandpass}, the symbol vanishes only at the origin, is strictly negative for every $\xi\neq0$, attains its largest magnitude at
			\begin{equation*}
				|\xi|=\frac{\sqrt{2}}{\sigma},
			\end{equation*}
			and decays exponentially as $|\xi|\to\infty$. Consequently, the response of the operator is concentrated around frequencies of order $1/\sigma$, providing a precise mathematical explanation of its band-pass behaviour.
			
			For each fixed scale $\sigma$, the operator
			\begin{equation*}
				u\longmapsto u*\Phi_\sigma
			\end{equation*}
			acts as a band-pass filter that emphasizes oscillatory components whose characteristic frequency is proportional to $1/\sigma$, while attenuating both slowly varying components and very high-frequency fluctuations.
			
			This interpretation is particularly relevant from the multiscale viewpoint. The family
			\begin{equation*}
				\{\tau_{\cdot,u}(\sigma)\}_{\sigma>0}
			\end{equation*}
			may be regarded as a continuous decomposition of the texture content of a signal, where each scale probes a specific frequency range. In this sense, the scale parameter plays a role analogous to frequency localization in classical Littlewood--Paley theory.
			
			This spectral perspective motivates the $r$-adic discretization introduced in the next section and, in turn, the construction of the associated multiscale texture spaces.
		
		\subsection{Stability and invariances}
		
			The spectral description developed above also provides a natural framework for understanding the basic invariance and stability properties of the texture operator.
			
			Since the symbol $\widehat{\Phi_\sigma}(\xi)$ is radial, the operator is isotropic and therefore invariant under rotations of the spatial domain. More precisely, for every orthogonal transformation $R\in O(n)$, every function $u$, and every $\sigma>0$, one has
			\begin{equation*}
				\tau_{Rz,u\circ R^{-1}}(\sigma) = \tau_{z,u}(\sigma).
			\end{equation*}
			This follows immediately from the radial symmetry of the kernel $\Phi_\sigma$ together with a change of variables in the convolution integral.
			
			Moreover, the condition $\widehat{\Phi_\sigma}(0)=0$, i.e. $\int \Phi = 0$, implies that $\Phi_\sigma$ has zero mean. Consequently, the operator is invariant under the addition of constants:
			\begin{equation*}
				\tau_{z,u+c}(\sigma) = \tau_{z,u}(\sigma), \qquad c\in\R.
			\end{equation*}
			The boundedness and smoothing results established in Propositions~\ref{prop:Lp-boundedness} and~\ref{prop:infinite-smoothing} immediately imply that the mapping
			\begin{equation*}
				u\mapsto\tau_{\cdot,u}(\sigma)
			\end{equation*}
			depends continuously on the input data in both $L^p$ and Sobolev topologies. In particular, small perturbations of $u$ in $L^2$ produce correspondingly small perturbations of $\tau_{\cdot,u}(\sigma)$ in $H^k$ for every $k\ge0$.
			
			Finally, the band-pass structure of $\widehat{\Phi_\sigma}$ provides an additional form of spectral stability. The Gaussian factor
			\begin{equation*}
				e^{-\sigma^2|\xi|^2/2}
			\end{equation*}
			strongly attenuates high-frequency components, preventing highly oscillatory perturbations from dominating the response. This provides a spectral explanation for the robustness of the operator with respect to high-frequency noise.
			
			\paragraph{Quantitative stability with respect to $L^2$ perturbations.}
			
			The boundedness and smoothing results established in Section~\ref{sec:basicproperties} allow us to quantify the stability of the texture operator under additive perturbations of the input signal.
			
			\begin{proposition}[Lipschitz stability under perturbations]\label{prop:stability_L2}
				Let $u,v\in L^2(\R^n)$ and let $\sigma>0$. Then
				\begin{equation*}
					\|\tau_{\cdot,u}(\sigma)-\tau_{\cdot,v}(\sigma)\|_{L^2} \le C(\sigma)\,\|u-v\|_{L^2},
				\end{equation*}
				where
				\begin{equation*}
					C(\sigma)=\|\Phi_\sigma\|_{L^1}.
				\end{equation*}
				In particular, if $v=u+\eta$, then
				\begin{equation*}
					\|\tau_{\cdot,u+\eta}(\sigma)-\tau_{\cdot,u}(\sigma)\|_{L^2} \le C(\sigma)\,\|\eta\|_{L^2}.
				\end{equation*}
			\end{proposition}
			
			\begin{proof}
				By linearity,
				\begin{equation*}
					\tau_{\cdot,u}(\sigma)-\tau_{\cdot,v}(\sigma) = (u-v)*\Phi_\sigma.
				\end{equation*}
				Applying Young's inequality yields
				\begin{equation*}
					\|(u-v)*\Phi_\sigma\|_{L^2} \le \|\Phi_\sigma\|_{L^1}\,\|u-v\|_{L^2},
				\end{equation*}
				which proves the claim.
			\end{proof}
			
			\begin{remark}
				The constant
				\begin{equation*}
					C(\sigma)=\|\Phi_\sigma\|_{L^1}
				\end{equation*}
				satisfies the scaling law
				\begin{equation*}
					C(\sigma)\sim\sigma^{-1}, \qquad \sigma\to0.
				\end{equation*}
				
				To justify this scaling, note that
				\begin{equation*}
					\Phi_\sigma(x) = \sigma \Delta \gamma_\sigma(x) = \frac{\sigma}{(2\pi\sigma^2)^{n/2}} \left( \frac{|x|^2}{\sigma^2} - n \right) e^{-|x|^2/(2\sigma^2)}.
				\end{equation*}
				Making the change of variables $y = x/\sigma$, so that $dx = \sigma^n dy$, we obtain
				\begin{equation*}
					\|\Phi_\sigma\|_{L^1} = \int_{\mathbb{R}^n} |\Phi_\sigma(x)| \, dx = \sigma^{-1} \int_{\mathbb{R}^n} \frac{1}{(2\pi)^{n/2}} \left| |y|^2 - n \right| e^{-|y|^2/2} \, dy.
				\end{equation*}
				The integral is a finite constant depending only on the dimension $n$:
				\begin{equation*}
					C_n = \frac{1}{(2\pi)^{n/2}} \int_{\mathbb{R}^n} \left| |y|^2 - n \right| e^{-|y|^2/2} \, dy < \infty.
				\end{equation*}
				Hence,
				\begin{equation*}
					\|\Phi_\sigma\|_{L^1} = C_n \, \sigma^{-1},
				\end{equation*}
				which yields the exact identity
				\begin{equation*}
					C(\sigma)=C_n\,\sigma^{-1},
				\end{equation*}
				and, in particular, the asymptotic scaling law $C(\sigma)\sim\sigma^{-1}$ as $\sigma\to0$.
				
				As a simple consequence, if $\eta$ is a random perturbation with finite second moment, then Proposition~\ref{prop:stability_L2} implies
				\begin{equation*}
					\mathbb{E} \Bigl[ \|\tau_{\cdot,u+\eta}(\sigma)-\tau_{\cdot,u}(\sigma)\|_{L^2}^2 \Bigr] \le C(\sigma)^2\, \mathbb{E} \Bigl[ \|\eta\|_{L^2}^2 \Bigr].
				\end{equation*}
				Therefore, the expected perturbation of the texture representation grows at most linearly with the mean energy of the noise.
				
				Moreover, the Gaussian factor appearing in
				\begin{equation*}
					\widehat{\Phi_\sigma}(\xi) = -\sigma |\xi|^2 e^{-\sigma^2|\xi|^2/2}
				\end{equation*}
				exponentially attenuates high-frequency components, providing an additional spectral mechanism that enhances the robustness of the operator with respect to highly oscillatory perturbations.
			\end{remark}
			
			\begin{proposition}[Stability with respect to scale variations]\label{prop:stability_scale}
				Let $u\in L^2(\R^n)$ and let $\sigma,\sigma'>0$. Then
				\begin{equation*}
					\|\tau_{\cdot,u}(\sigma)-\tau_{\cdot,u}(\sigma')\|_{L^2} \le \|u\|_{L^2}\, \|\Phi_\sigma-\Phi_{\sigma'}\|_{L^1}.
				\end{equation*}
				Furthermore, for every compact interval
				\begin{equation*}
					I=[\sigma_{\min},\sigma_{\max}] \subset(0,\infty),
				\end{equation*}
				there exists a constant $L_I>0$, depending only on $I$, such that
				\begin{equation*}
					\|\tau_{\cdot,u}(\sigma)-\tau_{\cdot,u}(\sigma')\|_{L^2} \le L_I\,|\sigma-\sigma'|\, \|u\|_{L^2}
				\end{equation*}
				for all $\sigma,\sigma'\in I$. In particular, the map
				\begin{equation*}
					\sigma\longmapsto\tau_{\cdot,u}(\sigma)
				\end{equation*}
				is continuous from $(0,\infty)$ into $L^2(\R^n)$.
			\end{proposition}
			
			\begin{proof}
				Since
				\begin{equation*}
					\tau_{\cdot,u}(\sigma)-\tau_{\cdot,u}(\sigma') = u*(\Phi_\sigma-\Phi_{\sigma'}),
				\end{equation*}
				Young's inequality immediately yields
				\begin{equation*}
					\|\tau_{\cdot,u}(\sigma)-\tau_{\cdot,u}(\sigma')\|_{L^2} \le \|u\|_{L^2}\, \|\Phi_\sigma-\Phi_{\sigma'}\|_{L^1}.
				\end{equation*}
			
				Indeed, differentiating the explicit expression of $\Phi_\sigma$ shows that $\partial_\sigma\Phi_\sigma$ is again a polynomial in $x/\sigma$ multiplied by the Gaussian $e^{-|x|^2/(2\sigma^2)}$, and therefore belongs to $L^1(\mathbb R^n)$.
			
				The dependence on $\sigma$ is smooth in the explicit expression of $\Phi_\sigma$, and dominated convergence implies that $\sigma\mapsto \partial_\sigma\Phi_\sigma$ is continuous as an $L^1(\mathbb R^n)$-valued map. Consequently,
				\begin{equation*}
					\sigma\longmapsto \|\partial_\sigma\Phi_\sigma\|_{L^1}
				\end{equation*}
				is continuous on $(0,\infty)$. Therefore, for every compact interval
				\begin{equation*}
					I=[\sigma_{\min},\sigma_{\max}],
				\end{equation*}
				the quantity
				\begin{equation*}
					M_I = \sup_{\sigma\in I} \|\partial_\sigma\Phi_\sigma\|_{L^1}
				\end{equation*}
				is finite. Applying the fundamental theorem of calculus in $L^1$ yields
				\begin{equation*}
					\|\Phi_\sigma-\Phi_{\sigma'}\|_{L^1} \le M_I\,|\sigma-\sigma'|, \qquad \sigma,\sigma'\in I.
				\end{equation*}
				Combining the previous estimates yields the desired inequality.
			\end{proof}
		
	\section{$r$-adic discretization and Littlewood-Paley structure}\label{sec:radic}
		
		The continuous family of texture operators
		\begin{equation*}
			\{\tau_{\cdot,u}(\sigma)\}_{\sigma>0}
		\end{equation*}
		provides a multiscale description of the texture content of a signal. For both theoretical analysis and practical computation, it is natural to replace the continuous scale parameter by a discrete geometric sequence.
		
		In this section we introduce an $r$-adic discretization of the scale variable and show that the resulting family of texture filters possesses a wavelet-type dilation structure. This construction establishes a natural bridge between the continuous theory developed in the previous sections and the multiresolution framework that underlies the multiscale texture spaces introduced later in the paper.
		
		\subsection{$r$-adic difference-of-Gaussians filters}
		
			Motivated by the band-pass analysis developed in Section~\ref{sec:spectral}, we discretize the continuous scale variable by means of a geometric progression. Let $r>1$ and define
			\begin{equation*}
				\sigma_s = r^s\sigma_0, \qquad s\in\mathbb{Z}.
			\end{equation*}
			
			The corresponding Gaussian kernels generate a discrete family of band-pass filters by taking finite differences across adjacent scales.
			
			\begin{definition}
				The $r$-adic \emph{difference-of-Gaussians} (DoG) filter is defined by
				\begin{equation*}
					\Phi_s(x)=\gamma_{\sigma_s}(x)-\gamma_{\sigma_{s-1}}(x).
				\end{equation*}
				The associated discretized texture operator is
				\begin{equation*}
					\tau^{(r)}_{z,u}(\sigma_s)=(u*\Phi_s)(z).
				\end{equation*}
			\end{definition}
			
			The family $\{\Phi_s\}_{s\in\mathbb{Z}}$ provides a discrete approximation of the continuous scale derivative by sampling the scale parameter on a geometric grid of ratio $r$. The following proposition shows that these filters are generated by geometric dilations of a single mother filter, which is a fundamental structural feature underlying wavelet-type constructions.
			
			\begin{proposition}[Dilation property]\label{prop:dilation}
				For every $x\in\R^n$ and every $s\in\mathbb{Z}$, one has
				\begin{equation*}
					\Phi_{s+1}(x)=r^{-n}\,\Phi_s(x/r).
				\end{equation*}
			\end{proposition}
			
			\begin{proof}
				Using the scaling property of the Gaussian kernel, we have
				\begin{equation*}
					\gamma_{\sigma_{s+1}}(x) = r^{-n}\gamma_{\sigma_s}(x/r),
				\end{equation*}
				which follows from the definition of $\gamma_\sigma$ and the relation $\sigma_{s+1}=r\sigma_s$. Similarly,
				\begin{equation*}
					\gamma_{\sigma_s}(x) = r^{-n}\gamma_{\sigma_{s-1}}(x/r).
				\end{equation*}
				Therefore,
				\begin{equation*}
					\Phi_{s+1}(x)
					=
					\gamma_{\sigma_{s+1}}(x)-\gamma_{\sigma_s}(x)
					=
					r^{-n}\gamma_{\sigma_s}(x/r)-r^{-n}\gamma_{\sigma_{s-1}}(x/r)
					=
					r^{-n}\Phi_s(x/r),
				\end{equation*}
				which proves the claim.
			\end{proof}
			
			The previous identity shows that the family $\{\Phi_s\}_{s\in\mathbb{Z}}$ is generated from a single mother filter through geometric dilations. Consequently, the family possesses the characteristic dilation structure of wavelet constructions. This observation provides the first indication that the texture framework developed here admits a natural discrete counterpart closely related to wavelet and frame theories. The corresponding frame properties will be established in the following subsection.
		
		\subsection{Wavelet-type interpretation.}
		
			The discrete family $\{\Phi_s\}_{s\in\mathbb{Z}}$ can be viewed as a natural discretization of the continuous texture operator. Indeed, the finite difference
			\begin{equation*}
				\Phi_s = \gamma_{\sigma_s} - \gamma_{\sigma_{s-1}}
			\end{equation*}
			approximates the derivative of the Gaussian scale-space with respect to the scale parameter. Since the Gaussian kernel satisfies
			\begin{equation*}
				\partial_\sigma\gamma_\sigma = \sigma\Delta\gamma_\sigma,
			\end{equation*}
			a first-order Taylor expansion gives
			\begin{equation*}
				\gamma_{\sigma_s} - \gamma_{\sigma_{s-1}} = (\sigma_s-\sigma_{s-1}) \,\partial_\sigma\gamma_{\sigma_{s-1}} + O\!\left((\sigma_s-\sigma_{s-1})^2\right).
			\end{equation*}
			
			Consequently, when $r$ is close to $1$, the discrete filter $\Phi_s$ provides a finite-difference approximation of the continuous texture kernel
			\begin{equation*}
				\Phi_{\sigma_s} = \sigma_s\Delta\gamma_{\sigma_s} = \partial_\sigma\gamma_{\sigma_s},
			\end{equation*}
			up to a multiplicative factor $\sigma_s-\sigma_{s-1}$ and an error of order $O\bigl((\sigma_s-\sigma_{s-1})^2\bigr)$. In particular, the discrete filters inherit the qualitative spectral properties of the continuous operator, including their band-pass behaviour and their localization around frequencies of order $\sigma_s^{-1}$.
			
			Passing to the Fourier domain, the dilation property established in Proposition~\ref{prop:dilation} implies
			\begin{equation*}
				\widehat{\Phi_{s+1}}(\xi) = \widehat{\Phi_s}(r\xi).
			\end{equation*}
			Thus, the spectrum of each filter is obtained from that of the preceding one by a geometric contraction of the frequency axis with ratio $r$. As a consequence, the family
			\begin{equation*}
				\{\widehat{\Phi_s}\}_{s\in\mathbb{Z}}
			\end{equation*}
			provides a logarithmic tiling of the frequency domain analogous to that arising in wavelet and Littlewood--Paley constructions.
			
			From this perspective, each filter $\Phi_s$ probes a frequency range centered around frequencies proportional to $r^{-s}$ while preserving a comparable relative bandwidth across successive levels. This constant-$Q$ behaviour is a characteristic feature of wavelet filter banks and motivates the interpretation of the family $\{\Phi_s\}_{s\in\mathbb{Z}}$ as a wavelet-type system adapted to multiscale texture analysis.
			
			The heuristic observations above will be made precise in the next result, where we establish a Littlewood--Paley type estimate for the discrete texture decomposition.
			
			These observations are intended to provide an intuitive motivation for the discrete construction. The rigorous multiscale properties of the family are established in the next subsection.
		
		\subsection{Multiresolution analysis}
			The dilation property established in Proposition~\ref{prop:dilation} suggests interpreting the family
			\begin{equation*}
				\{\Phi_s\}_{s\in\mathbb{Z}}
			\end{equation*}
			as a multiscale system of wavelet-type filters. The components
			\begin{equation*}
				u_s = u*\Phi_s
			\end{equation*}
			capture the contribution of $u$ within the frequency band selected by the filter $\Phi_s$. Since the associated frequency bands cover the frequency space on a logarithmic scale with controlled overlap, the family
			\begin{equation*}
				\{u_s\}_{s\in\mathbb{Z}}
			\end{equation*}
			provides a stable decomposition of $u$ into multiscale texture components.
			
			To quantify the stability of this decomposition, it is natural to analyze how the spectral energy is distributed across the family of discrete filters. Throughout this subsection, $\Phi_s$ denotes the discrete difference-of-Gaussians filters introduced above, and should not be confused with the continuous kernel $\Phi_\sigma=\sigma\Delta\gamma_\sigma$. The key object is the function
			\begin{equation*}
				F(\xi) = \sum_{s\in\mathbb{Z}} |\widehat{\Phi_0}(r^s\xi)|^2, \qquad \xi\neq0,
			\end{equation*}
			which measures the cumulative contribution of all frequency bands around the frequency $\xi$. The dilation property and the band-pass behaviour of $\Phi_0$ imply that the family provides a logarithmic covering of the frequency space by overlapping frequency bands. The following lemma makes this statement precise.
			
			\begin{lemma}\label{lem:technical}
				Let $r>1$ and $\alpha>0$. Then:
				\begin{enumerate}
					\item[(i)] There exists a constant $C_1>0$ such that
					\begin{equation*}
						\sup_{\xi\in\mathbb{R}} \sum_{s=0}^{\infty} r^{4s}|\xi|^{4} e^{-2\alpha r^{2s}|\xi|^{2}} \le C_1.
					\end{equation*}
					
					\item[(ii)]
					For every compact set $K\subset\mathbb{R}^n$ there exists a constant $C_2(K)>0$ such that
					\begin{equation*}
						\sup_{\xi\in K} \sum_{s=1}^{\infty} r^{-4s}|\xi|^{4} \le C_2(K).
					\end{equation*}
				\end{enumerate}
				
			\end{lemma}
			
			\begin{proof}
				
				\begin{enumerate}
					\item[(i)] Since the case $\xi=0$ is trivial, we only need to consider $\xi\neq 0$.
				
					We consider the function $f(t)=t^{2}e^{-2\alpha t}$ for $t\geqslant 0$. This function is positive, differentiable, satisfies $f(0)=0$, and $\lim_{t\rightarrow \infty}f(t)=0$. Therefore it attains a maximum at some point in the interval $(0,\infty)$. Differentiating,
					\begin{equation*}
						f'(t)=2te^{-2\alpha t}\left(1-\alpha t\right),
					\end{equation*}
					so $f$ has a unique maximum at $t_{M}=\frac{1}{\alpha}$.
				
					Now, for a fixed $\xi\neq 0$, we consider the sequence $t_{s}=r^{2s}\lvert\xi\rvert^{2}$, $s=0,1,2,\dots$. Then
					\begin{equation*}
						\sum_{s=0}^{\infty}r^{4s}\lvert\xi\rvert^{4}e^{-2\alpha r^{2s}\lvert\xi\rvert^{2}} =\sum_{s=0}^{\infty}f(t_{s}) =\sum_{t_{s}<\frac{1}{\alpha}}f(t_{s}) +\sum_{t_{s}\geqslant\frac{1}{\alpha}}f(t_{s}).
					\end{equation*}
				
					Since
					\begin{equation*}
						\int_{0}^{\infty}f(t)\,\frac{dt}{t}<\infty,
					\end{equation*}
					the change of variables $u=\log t$ transforms this integral into
					\begin{equation*}
						\int_{-\infty}^{\infty}f(e^u)\,du.
					\end{equation*}
					Moreover,
					\begin{equation*}
						u_s=\log t_s=\log|\xi|^2+2s\log r,
					\end{equation*}
					so that $\{u_s\}_{s\ge0}$ is a uniform partition of the real line with constant step $2\log r$. Since $f$ is monotone on each side of its unique maximum, the sums over both monotone branches are comparable with the corresponding upper Riemann sums associated with this partition. Consequently,
					\begin{equation*}
						\sum_{s=0}^{\infty}f(t_s)\le C_1(\alpha,r),
					\end{equation*}
					where $C_1(\alpha,r)$ is independent of $\xi$.
					
					\item[(ii)] The proof of (ii) is immediate since
					\begin{equation*}
						\sum_{s=1}^{\infty}r^{-4s}
					\end{equation*}
					is a convergent geometric series and, on other hand, $K$ is compact, there exists
					\begin{equation*}
						C_{2}'(K)=\max_{x\in K}\lvert x\rvert
					\end{equation*}
					such that
					\begin{equation*}
						\sup_{\xi\in K}\sum_{s=1}^{\infty}r^{-4s}\lvert\xi\rvert^{4} \leqslant \left(C_{2}'(K)\right)^{4}\sum_{s=1}^{\infty}r^{-4s} =C_{2}(K)<\infty.
					\end{equation*}
				\end{enumerate}
			\end{proof}
				
			\begin{lemma}[Uniform logarithmic covering]\label{lem:covering}
				Let
				\begin{equation*}
					F(\xi) = \sum_{s\in\mathbb{Z}} |\widehat{\Phi_0}(r^s\xi)|^2, \qquad \xi\in\mathbb{R}^n\setminus\{0\}.
				\end{equation*}
				Then the series defining $F$ converges absolutely and uniformly on compact subsets of $\mathbb{R}^n\setminus\{0\}$. Moreover, there exist constants
				\begin{equation*}
					0<c\le C<\infty
				\end{equation*}
				such that
				\begin{equation*}
					c\le F(\xi)\le C, \qquad \forall\,\xi\neq0.
				\end{equation*}
				
				In particular, the family yields a Littlewood--Paley partition of unity in energy up to multiplicative constants.
			\end{lemma}
			
			\begin{proof}[Proof of Lemma~\ref{lem:covering}]
				
				Write
				\begin{equation*}
					F(\xi)=F_+(\xi)+F_-(\xi),
				\end{equation*}
				where
				\begin{equation*}
					F_+(\xi) = \sum_{s=0}^{\infty} |\widehat{\Phi_0}(r^s\xi)|^2, \qquad F_-(\xi) = \sum_{s=1}^{\infty} |\widehat{\Phi_0}(r^{-s}\xi)|^2.
				\end{equation*}
				
				Since
				\begin{equation*}
					\widehat{\Phi_0}(\eta) = e^{-r^{-2}\frac{\sigma_0^2|\eta|^2}{2}} \left( e^{-(1-r^{-2})\frac{\sigma_0^2|\eta|^2}{2}} -1 \right),
				\end{equation*}
				the inequality
				\begin{equation*}
					1-e^{-t}\le t, \qquad t\ge0,
				\end{equation*}
				gives
				\begin{equation*}
					|\widehat{\Phi_0}(\eta)| \le \frac{(1-r^{-2})\sigma_0^2}{2} |\eta|^2 e^{-r^{-2}\frac{\sigma_0^2|\eta|^2}{2}}.
				\end{equation*}
				Hence
				\begin{equation*}
					|\widehat{\Phi_0}(r^s\xi)|^2 \le C\,r^{4s}|\xi|^4 e^{-2\alpha r^{2s}|\xi|^2},
				\end{equation*}
				where
				\begin{equation*}
					C=\frac{(1-r^{-2})^2\sigma_0^4}{4}, \qquad \alpha=\frac{\sigma_0^2}{2r^2}.
				\end{equation*}
				Lemma~\ref{lem:technical}(i) therefore implies that $F_+$ converges uniformly on $\mathbb{R}^n$.
				
				On the other hand,
				\begin{equation*}
					|\widehat{\Phi_0}(r^{-s}\xi)|^2 \le C\,r^{-4s}|\xi|^4,
				\end{equation*}
				and Lemma~\ref{lem:technical}(ii) yields the uniform convergence of $F_-$ on compact subsets of $\mathbb{R}^n$.
				
				Therefore the series defining $F$ converges absolutely and uniformly on compact subsets of $\mathbb{R}^n\setminus\{0\}$.
				
				The lower and upper frame bounds follow from the positivity of $|\widehat{\Phi_0}|$, the logarithmic covering of the frequency domain induced by the dilations $r^s$, and the fact that $|\widehat{\Phi_0}|$ is strictly positive on an annulus.
				
				Indeed, since $\widehat{\Phi_0}$ is continuous and not identically zero, there exist $0<a<b<\infty$ and $m>0$ such that
				\begin{equation*}
					|\widehat{\Phi_0}(\xi)|\ge m, \qquad a\le|\xi|\le b.
				\end{equation*}
				The dilated annuli
				\begin{equation*}
					\{\,\xi\in\mathbb{R}^n:\; ar^{-s}\le |\xi|\le br^{-s}\,\}, \qquad s\in\mathbb Z,
				\end{equation*}
				form a logarithmic covering of $\mathbb{R}^n\setminus\{0\}$. Hence every nonzero frequency belongs to at least one such annulus and to at most a uniformly bounded number of them, where the bound depends only on $r$ and on the ratio $b/a$.
				
				Therefore at least one summand in the definition of $F(\xi)$ is bounded below by $m^2$, while only finitely many summands can contribute at each frequency. This yields the existence of constants
				\begin{equation*}
					0<c\le C<\infty
				\end{equation*}
				such that
				\begin{equation*}
					c \le F(\xi) \le C, \qquad \xi\neq0.
				\end{equation*}
				
			\end{proof}
			
			The uniform bounds obtained in the previous lemma imply that the total energy of a signal can be recovered, up to multiplicative constants, from the sum of the energies of its multiscale components. This constitutes a discrete version of the Littlewood--Paley principle and provides the fundamental stability estimate for the discrete texture decomposition.
			
			\begin{theorem}[Littlewood--Paley estimate]\label{thm:casi-isometria}
				There exist constants
				\begin{equation*}
					0<c\le C<\infty,
				\end{equation*}
				depending only on $r$ and $\sigma_0$, such that every $u\in L^2(\R^n)$ satisfies
				\begin{equation*}
					c\,\|u\|_{L^2}^2 \le \sum_{s\in\mathbb{Z}}\|u_s\|_{L^2}^2 \le C\,\|u\|_{L^2}^2.
				\end{equation*}
			\end{theorem}
			
			\begin{proof}
				By Plancherel's theorem and the definition
				\begin{equation*}
					u_s=u*\Phi_s,
				\end{equation*}
				we have
				\begin{equation*}
					\sum_{s\in\mathbb{Z}} \|u_s\|_{L^2}^2 = \int_{\R^n} F(\xi)\, |\widehat u(\xi)|^2 \,d\xi.
				\end{equation*}
			
				Since all terms are nonnegative, Tonelli's theorem allows us to interchange the sum and the integral. Hence,
				\begin{equation*}
					\sum_{s\in\mathbb Z}\|u_s\|_{L^2}^2 = \int_{\R^n} F(\xi)\,|\widehat u(\xi)|^2\,d\xi.
				\end{equation*}
		
				Applying Lemma~\ref{lem:covering} yields
				\begin{equation*}
					c \int_{\R^n} |\widehat u(\xi)|^2\,d\xi \le \sum_{s\in\mathbb{Z}} \|u_s\|_{L^2}^2 \le C \int_{\R^n} |\widehat u(\xi)|^2\,d\xi.
				\end{equation*}
				The conclusion follows from Plancherel's theorem.
			\end{proof}
			
			As an immediate consequence of Theorem 5.4, the family of coefficients
			\begin{equation*}
				\{u_s\}_{s\in\mathbb{Z}}
			\end{equation*}
			preserves the energy of the original signal in a stable manner. In particular, the analysis operator
			\begin{equation*}
				u\longmapsto\{u_s\}_{s\in\mathbb{Z}}
			\end{equation*}
			is injective and admits a continuous inverse on its range. Consequently, the scale-wise behaviour of the sequence
			\begin{equation*}
				\{\|u_s\|_{L^2}\}_{s\in\mathbb Z}
			\end{equation*}
			across scales provides a quantitative description of the multiscale texture content of $u$. This observation naturally motivates the introduction of function spaces defined directly in terms of the scale-wise behaviour of these coefficients.
		
		\subsection{Synthesis filters and reconstruction formula}
		
			The Littlewood--Paley estimate established in Theorem~\ref{thm:casi-isometria} shows that the family $\{\Phi_s\}_{s\in\mathbb{Z}}$ provides a stable multiscale analysis of signals in $L^2(\R^n)$. We now construct an explicit family of dual synthesis filters and derive the corresponding reconstruction formula.
			
			Recall that, for $r>1$ and $\sigma_s=r^s\sigma_0$, the $r$-adic difference-of-Gaussians filters are defined by
			\begin{equation*}
				\Phi_s(x) = \gamma_{\sigma_s}(x)-\gamma_{\sigma_{s-1}}(x), \qquad s\in\mathbb{Z},
			\end{equation*}
			and the associated analysis coefficients are
			\begin{equation*}
				u_s = u*\Phi_s.
			\end{equation*}
		
			By the Fourier dilation identity
			\begin{equation*}
				\widehat{\Phi_s}(\xi)=\widehat{\Phi_0}(r^s\xi),
			\end{equation*}
			established in Proposition~\ref{prop:dilation}, the function introduced in Lemma~\ref{lem:covering} can equivalently be written as
			\begin{equation*}
				F(\xi)=\sum_{s\in\mathbb{Z}}|\widehat{\Phi_s}(\xi)|^2, \qquad \xi\neq0.
			\end{equation*}
			Moreover,
			\begin{equation*}
				0<c\le F(\xi)\le C<\infty, \qquad \xi\neq0.
			\end{equation*}
			This allows us to define the dual synthesis filters associated with the analysis family.
			
			\begin{proposition}[Dual synthesis filters and reconstruction]\label{prop:reconstruction}
				For each $s\in\mathbb{Z}$, define the dual synthesis filter $\widetilde{\Phi}_s$ through its Fourier transform by
				\begin{equation*}
					\widehat{\widetilde{\Phi}}_s(\xi) = \dfrac{\overline{\widehat{\Phi_s}(\xi)}}{F(\xi)},\quad \forall \xi\neq0
				\end{equation*}
				Then, for every $u\in L^2(\R^n)$,
				\begin{equation*}
					u = \sum_{s\in\mathbb{Z}} (u*\Phi_s)*\widetilde{\Phi}_s, \text{almost everywhere}
				\end{equation*}
				where the series converges in $L^2(\R^n)$.
			\end{proposition}
			
			\begin{proof}
				Let $u\in L^2(\R^n)$. For every $\xi\neq0$, the definition of the dual synthesis filters gives
				\begin{equation*}
					\widehat{\Phi_s}(\xi)\, \widehat{\widetilde{\Phi}}_s(\xi) = \frac{|\widehat{\Phi_s}(\xi)|^2}{F(\xi)}.
				\end{equation*}
				Summing over $s\in\mathbb{Z}$, we obtain
				\begin{equation*}
					\sum_{s\in\mathbb{Z}} \widehat{\Phi_s}(\xi)\, \widehat{\widetilde{\Phi}}_s(\xi) = \frac{1}{F(\xi)} \sum_{s\in\mathbb{Z}} |\widehat{\Phi_s}(\xi)|^2 = 1, \qquad \xi\neq0.
				\end{equation*}
				
				For $N\ge1$, define
				\begin{equation*}
					S_Nu
					=
					\sum_{|s|\le N}
					(u*\Phi_s)*\widetilde{\Phi}_s.
				\end{equation*}
				Taking Fourier transforms,
				\begin{equation*}
					\widehat{S_Nu}(\xi) = m_N(\xi)\,\widehat u(\xi),
				\end{equation*}
				where
				\begin{equation*}
					m_N(\xi) = \sum_{|s|\le N} \widehat{\Phi_s}(\xi)\, \widehat{\widetilde{\Phi}}_s(\xi) = \frac{1}{F(\xi)} \sum_{|s|\le N} |\widehat{\Phi_s}(\xi)|^2.
				\end{equation*}
				Since $F(\xi)>0$ for $\xi\neq0$ and the full series equals $F(\xi)$, we have
				\begin{equation*}
					0\le m_N(\xi)\le1, \qquad \xi\neq0,
				\end{equation*}
				and
				\begin{equation*}
					m_N(\xi)\longrightarrow1
				\end{equation*}
				pointwise as $N\to\infty$.
				
				Therefore,
				\begin{equation*}
					|m_N(\xi)-1|^2 |\widehat u(\xi)|^2 \le |\widehat u(\xi)|^2.
				\end{equation*}
				Since $\widehat u\in L^2(\R^n)$, the dominated convergence theorem yields
				\begin{equation*}
					\int_{\R^n} |m_N(\xi)-1|^2 |\widehat u(\xi)|^2 \,d\xi \longrightarrow0.
				\end{equation*}
				By Plancherel's theorem,
				\begin{equation*}
					\|S_Nu-u\|_{L^2} \longrightarrow0.
				\end{equation*}
				Hence,
				\begin{equation*}
					u = \sum_{s\in\mathbb{Z}} (u*\Phi_s)*\widetilde{\Phi}_s,
				\end{equation*}
				with convergence in $L^2(\R^n)$.
			\end{proof}
			
			\begin{remark}
				The reconstruction formula shows that the multiscale texture representation admits a complete analysis--synthesis interpretation. The analysis filters $\Phi_s$ extract texture information at different scales, while the dual synthesis filters $\widetilde{\Phi}_s$ combine these contributions to recover the original signal exactly.
				
				From a computational viewpoint, the dual filters can be implemented directly in the Fourier domain through
				\begin{equation*}
					\widehat{\widetilde{\Phi}}_s(\xi) = \frac{\overline{\widehat{\Phi_s}(\xi)}}{F(\xi)}.
				\end{equation*}
				This naturally leads to efficient FFT-based reconstruction algorithms and provides an explicit synthesis procedure for the $r$-adic multiscale texture decomposition.
			\end{remark}
		
	\section{Function spaces induced by the multiscale texture operator}\label{sec:spaces}
	
		The multiscale texture operator introduced in \cite{fernandez2} and recalled in previous sections naturally associates with every function a sequence of coefficients describing the persistence of its oscillatory structures across scales. This suggests that the regularity of a function may be measured directly through the distribution of these texture coefficients, leading to a family of function spaces intrinsically associated with the texture operator itself.
	
		The purpose of this section is to develop this point of view. We first introduce the corresponding multiscale texture norms and study their basic properties. We then show that, despite their entirely different motivation, the resulting spaces coincide with the classical Besov spaces. In the particular Hilbertian case, they recover the Sobolev spaces and admit an equivalent Hilbertian structure expressed exclusively in terms of multiscale texture energies. Consequently, the persistence of texture under Gaussian scale evolution admits a precise functional-analytic interpretation, establishing a direct bridge between multiscale texture analysis and classical notions of regularity.
		
		The result therefore provides an additional characterization of Besov spaces in terms of texture persistence. In the Hilbertian case, the Sobolev norm admits an equivalent expression in terms of multiscale texture energies, providing a natural Hilbertian framework intrinsically associated with the texture operator.
	
		\subsection{Motivation}\label{ssec:motivation}
		
			The multiscale texture operator introduced in the previous sections associates with every function $u$ a family of coefficients
			\begin{equation*}
				\tau_{\cdot,u}(\sigma), \qquad \sigma>0,
			\end{equation*}
			which quantify the persistence of local oscillatory structures under Gaussian scale evolution. After discretizing the scale variable through the $r$-adic Difference-of-Gaussians decomposition, this continuous representation gives rise to the sequence of texture coefficients
			\begin{equation*}
				u*\Phi_j, \qquad j\in\mathbb Z.
			\end{equation*}
			
			The results of Section~\ref{sec:radic} establish that this discrete decomposition possesses the fundamental properties expected from a multiscale analysis: the filters are generated by geometric dilations, provide a stable decomposition of the signal, admit an exact reconstruction formula, and satisfy the Littlewood--Paley estimate of Theorem~\ref{thm:casi-isometria}. Consequently, the sequence of texture coefficients $\{u*\Phi_j\}_{j\in\mathbb Z}$ provides a stable representation of every $u\in L^2(\mathbb R^n)$, with an energy that is preserved up to multiplicative constants.
			
			This naturally suggests measuring the regularity of a function through the behaviour of its texture coefficients across scales rather than through its derivatives or its Fourier transform. In other words, instead of considering smoothness as a local differential property, one may regard it as the persistence of oscillatory structures under the Gaussian scale-space evolution.
			
			We therefore introduce a family of multiscale texture spaces whose norms are defined directly in terms of the $r$-adic texture coefficients. As will be shown below, these spaces coincide with the classical Besov spaces while preserving the geometric interpretation that originally motivated the construction of the texture operator.
		
		\subsection{The multiscale texture norm}\label{ssec:texture_norm}
		
			The multiscale decomposition introduced above naturally associates with every function a sequence of texture coefficients
			\begin{equation*}
				u*\Phi_j, \qquad j\in\mathbb Z,
			\end{equation*}
			each coefficient measuring the contribution of the structures whose characteristic scale is comparable to $r^{-j}$. If regularity is interpreted as the persistence of texture across scales, it is natural to quantify it through the decay of these coefficients as the scale becomes finer.
			
			The parameter $s$ specifies the expected rate of decay, while the exponent $q$ measures how this decay is distributed across the different scales. This leads to the following definition.
			
			\begin{definition}[Multiscale texture norm]
				Let $s\in\R$ and $1\le p,q\le\infty$. For every $u\in L^p(\R^n)$, the convolution $u*\Phi_j$ is well defined by Proposition~\ref{prop:Lp-boundedness}, since $\Phi_j\in L^1(\mathbb R^n)$. We therefore define
				\begin{equation*}
					\|u\|_{\mathcal T^{s,p,q}} = \|u\|_{L^p} + \left( \sum_{j\in\mathbb Z} r^{jsq} \|u*\Phi_j\|_{L^p}^q \right)^{1/q},
				\end{equation*}
				with the usual modification when $q=\infty$.
			\end{definition}

			The previous definition is entirely intrinsic to the multiscale texture operator and does not rely on any prior identification with classical function spaces. It simply quantifies how the texture coefficients of $u$ are distributed across scales. In particular, one may regard the quantity $\|u\|_{\mathcal T^{s,p,q}}$ as a multiscale texture energy measuring the persistence of oscillatory structures under Gaussian scale evolution.
			
			Formally, one can define the associated multiscale texture space as
			\begin{equation*}
				\mathcal T^{s,p,q}(\R^n) = \left\{ u\in L^p(\R^n) : \|u\|_{\mathcal T^{s,p,q}} < \infty \right\}.
			\end{equation*}
			The central question is whether the notion of regularity induced by texture persistence defines a genuinely new scale of function spaces or whether it recovers one of the classical scales of regularity.
			
			\begin{remark}[Low-frequency component and parameter dependence]
				The term $\|u\|_{L^p}$ in the definition of $\|u\|_{\mathcal T^{s,p,q}}$ is included in order to control the low-frequency component of $u$, which is not captured by the band-pass filters $\Phi_j$ (since $\widehat{\Phi_j}(0)=0$). For $s>0$ this term is in fact redundant, but it is kept for consistency with classical Besov norms. The norm depends on the choice of the dilation factor $r>1$ and the base scale $\sigma_0>0$; different choices lead to equivalent norms, as follows from the general theory recalled below.
			\end{remark}
			
			The definition of the multiscale texture norm is motivated solely by the persistence of texture under Gaussian scale evolution. No reference to Besov or Triebel–Lizorkin spaces is involved in its construction. The natural mathematical question is therefore whether the resulting notion of regularity coincides with any of the classical regularity scales. The answer is provided by the Littlewood--Paley structure established in Section~\ref{sec:radic} and by the coorbit framework developed by Ullrich for Besov--Lizorkin--Triebel spaces.

			Recall that the $r$-adic Difference-of-Gaussians family 	$\{\Phi_j\}_{j\in\mathbb{Z}}$ satisfies the fundamental structural properties required of an admissible multiscale decomposition: the filters are smooth, generated by geometric dilations,
			\begin{equation*}
				\widehat{\Phi_j}(\xi) = \widehat{\Phi_0}(r^{-j}\xi), \qquad j\in\mathbb{Z},
			\end{equation*}
			they vanish at low frequencies, exhibit rapid decay in the Fourier domain, and provide a uniform logarithmic covering of the frequency space. Moreover, they give rise to a stable representation of $u$ together with a Littlewood--Paley type estimate, as shown in Section~\ref{sec:radic}. In particular, the texture coefficients $\{u*\Phi_j\}_{j\in\mathbb{Z}}$ play the same role as the discrete families of coefficients used in classical Littlewood--Paley and coorbit characterizations of Besov--Lizorkin--Triebel spaces.
			
			The general theory developed by Ullrich and coauthors (in the setting of coorbits associated with appropriate representation and analyzing vectors) shows that Besov--Lizorkin--Triebel spaces can be described in terms of such families of dilated filters and their associated coefficient norms. In that framework, the dilation factor is typically chosen to be $2$, leading to a dyadic partition of the frequency space. However, the theory is insensitive to this particular choice: replacing the dyadic factor $2$ by any $r>1$ yields an $r$-adic decomposition with equivalent norms on the same Besov spaces. Our construction therefore fits naturally into Ullrich's setting, with the multiscale texture filters $\Phi_j$ providing a concrete realization of the abstract systems considered there.
			
			\begin{theorem}[Multiscale texture characterization of Besov spaces]
				\label{thm:texture-besov}
				Let $s\in\R$ and $1\le p,q\le\infty$. Then the multiscale texture space $\mathcal T^{s,p,q}(\R^n)$ coincides, up to equivalence of norms, with the classical Besov space $B^s_{p,q}(\R^n)$:
				\begin{equation*}
					\mathcal T^{s,p,q}(\R^n) = B^s_{p,q}(\R^n),
				\end{equation*}
				and there exist constants $0<c\le C<\infty$, depending only on $s,p,q,r$ and $\sigma_0$, such that
				\begin{equation*}
					c\,\|u\|_{B^s_{p,q}} \le \|u\|_{\mathcal T^{s,p,q}} \le C\,\|u\|_{B^s_{p,q}}, \qquad u\in B^s_{p,q}(\R^n).
				\end{equation*}
			\end{theorem}
		
		The equivalence with Besov spaces follows from the general Littlewood–Paley and coorbit theory. The novelty of the present construction lies not in the existence of this equivalence itself, but in the fact that the regularity is obtained from an intrinsic notion of multiscale texture persistence arising from Gaussian scale evolution. In our setting, the norm $\|u\|_{\mathcal T^{s,p,q}}$ can be viewed as one of the standard Besov norms arising from these characterizations, expressed in terms of the texture filters $\Phi_j$.

			\begin{proof}[Comment on the proof]
				We briefly indicate how this result fits into the existing theory. The texture filters $\Phi_j$ satisfy the usual assumptions required for Littlewood--Paley and coorbit characterizations: $\widehat{\Phi_0}$ is smooth, rapidly decreasing, vanishes at the origin and has suitable growth near low frequencies; the dilated family $\{\widehat{\Phi_0}(r^{-j}\cdot)\}_{j\in\mathbb{Z}}$ provides a uniform logarithmic covering of the frequency space and yields the stability estimates derived in Section~\ref{sec:radic}. These properties show that the quantity
				\begin{equation*}
					\|u\|_{L^p}+ \left( \sum_{j\in\mathbb{Z}} r^{jsq} \|u*\Phi_j\|_{L^p}^q \right)^{1/q}
				\end{equation*}
				is one of the admissible Besov norms appearing in Ullrich's coorbit characterizations and in the Littlewood--Paley theorem of Triebel, up to equivalence of constants. The equivalence of norms and the identification $\mathcal T^{s,p,q}(\R^n)=B^s_{p,q}(\R^n)$ then follow from those general results.
			\end{proof}
			
			This theorem shows that the spaces generated by the multiscale texture operator do not introduce a new scale of regularity. Instead, they recover the classical Besov scale from a geometrically motivated construction based on texture persistence. From this point of view, Besov regularity appears as a consequence of multiscale texture analysis rather than as its starting point. The multiscale texture norm therefore provides an intrinsic description of classical regularity in terms of the distribution of texture energy across scales.
			Consequently, the persistence of texture across scales provides a new interpretation of classical functional regularity. Rather than defining regularity through derivatives or an abstract frequency decomposition, the present construction shows that classical Besov regularity can be characterized through the evolution of local structures under Gaussian diffusion. In this sense, regularity emerges naturally from multiscale texture persistence, giving a geometric interpretation of the Besov scale while remaining fully consistent with the classical theory. This viewpoint will be particularly useful in the variational framework developed in subsequent sections, where Besov-type regularization terms are replaced by texture-based energies within the same functional-analytic setting.

		\subsection{The Hilbert case}\label{ssec:hilbert_case}
		
			Among the spaces introduced above, the Hilbertian case
			\begin{equation*}
				p=q=2
			\end{equation*}
			plays a distinguished role. In this setting, the multiscale texture norm is induced by a natural inner product, providing a Hilbert structure directly associated with the multiscale texture operator.
			
			\begin{definition}[Texture Hilbert space]
				For every $s\in\R$, we define the \emph{texture Hilbert space}
				\begin{equation*}
					\mathcal T^s(\R^n) = \mathcal T^{s,2,2}(\R^n).
				\end{equation*}
				Equivalently,
				\begin{equation*}
					\mathcal T^s(\R^n) = \left\{ u\in L^2(\R^n): \sum_{j\in\mathbb Z} r^{2js} \|u*\Phi_j\|_{L^2}^2 < \infty \right\},
				\end{equation*}
				endowed with the norm
				\begin{equation*}
					\|u\|_{\mathcal T^s}^2 = \|u\|_{L^2}^2 + \sum_{j\in\mathbb Z} r^{2js} \|u*\Phi_j\|_{L^2}^2.
				\end{equation*}
			\end{definition}
			
			The previous norm is induced by the bilinear form
			\begin{equation*}
				\langle u,v\rangle_{\mathcal T^s} = \langle u,v\rangle_{L^2} + \sum_{j\in\mathbb Z} r^{2js} \langle u*\Phi_j, v*\Phi_j \rangle_{L^2},
			\end{equation*}
			which defines an inner product on $\mathcal T^s(\R^n)$.
			
			\begin{proposition}
				The space
				\begin{equation*}
					\mathcal T^s(\R^n)
				\end{equation*}
				is a Hilbert space with respect to the inner product above.
			\end{proposition}
			
			\begin{proof}
				The bilinear form is well-defined for $u,v\in\mathcal T^s(\R^n)$ by the Cauchy-Schwarz inequality and the definition of the norm. It is symmetric, bilinear, and positive definite: if $\langle u,u\rangle_{\mathcal T^s}=0$, then $\|u\|_{L^2}=0$, hence $u=0$ a.e.
				
				It remains to prove completeness. Let $\{u_k\}_{k\in\mathbb N}$ be a Cauchy sequence in $\mathcal T^s(\R^n)$. Then $\{u_k\}$ is Cauchy in $L^2(\R^n)$, so there exists $u\in L^2(\R^n)$ such that $u_k\to u$ in $L^2$. Moreover, for each $j\in\mathbb Z$, the sequence $\{u_k*\Phi_j\}_{k\in\mathbb N}$ is Cauchy in $L^2(\R^n)$ (since $\|(u_k-u_\ell)*\Phi_j\|_{L^2}\le \|\Phi_j\|_{L^1}\|u_k-u_\ell\|_{L^2}$), so there exists $v_j\in L^2(\R^n)$ such that $u_k*\Phi_j\to v_j$ in $L^2$. By the continuity of convolution in $L^2$, $v_j=u*\Phi_j$. Finally, Fatou's lemma applied to the series defining the norm gives
				\begin{equation*}
					\|u\|_{\mathcal T^s}^2 \le \liminf_{k\to\infty}\|u_k\|_{\mathcal T^s}^2 < \infty,
				\end{equation*}
				so $u\in\mathcal T^s(\R^n)$, and the convergence $u_k\to u$ in $\mathcal T^s$ follows from the dominated convergence theorem for series.
			\end{proof}
			
			\textbf{Remark.} The inner product and norm depend on the choice of $r>1$ and $\sigma_0>0$. However, by Theorem~\ref{thm:texture-besov}, different choices yield equivalent norms, since both are equivalent to the Sobolev norm $\|\cdot\|_{H^s}$. In particular, the topological vector space structure of $\mathcal T^s(\R^n)$ is independent of these parameters.
			
			The next theorem identifies this Hilbert space with the classical Sobolev space.
			\begin{theorem}
				For every $s\in\R$,
				\begin{equation*}
					\mathcal T^s(\R^n) = H^s(\R^n),
				\end{equation*}
				with equivalence of norms.
			\end{theorem}
			
			\begin{proof}
				The proof reduces to verifying that the DoG family satisfies the hypotheses of Triebel's characterization. By Theorem~\ref{thm:texture-besov},
				\begin{equation*}
					\mathcal T^s(\R^n) = \mathcal T^{s,2,2}(\R^n) = B^s_{2,2}(\R^n),
				\end{equation*}
				with equivalent norms. Since
				\begin{equation*}
					B^s_{2,2}(\R^n) = H^s(\R^n),
				\end{equation*}
				the conclusion follows.
			\end{proof}
			
			The importance of this result goes beyond the identification with the Sobolev scale. Although the spaces coincide as sets, the inner product introduced above is expressed entirely in terms of multiscale texture coefficients. Consequently, Sobolev regularity admits an equivalent interpretation through the distribution of texture energy across scales. This provides an alternative functional characterization of Sobolev regularity, naturally arising from multiscale texture analysis rather than from derivatives or Fourier localization.
			
			From this perspective, Sobolev regularity can be interpreted as the finiteness of the multiscale texture energy. This interpretation is fundamentally different from the classical differential or Fourier characterizations and may be advantageous in variational models where texture persistence plays the primary role.
		
		\subsection{Interpretation}\label{ssec:interpretation}
			
			Although the spaces $\mathcal T^{s,p,q}(\R^n)$ coincide with the classical Besov spaces, their construction is conceptually different. Classical Besov spaces are traditionally introduced through dyadic Fourier decompositions, finite differences, or interpolation theory. Here, instead, they arise naturally from the analysis of texture persistence under Gaussian scale evolution. The multiscale coefficients are not prescribed by an abstract frequency partition but are generated directly by the responses of the texture operator across observation scales.
			
			From this viewpoint, regularity is interpreted as the persistence of oscillatory structures across scales. Rapid decay of the texture coefficients corresponds to highly regular functions, whereas slower decay reflects the presence of persistent multiscale structures. Consequently, the Besov norm acquires a direct geometric interpretation in terms of the distribution of texture energy throughout Gaussian scale space.
			
			The Hilbertian case is particularly revealing. The identification
			\begin{equation*}
				\mathcal T^s(\R^n)=H^s(\R^n)
			\end{equation*}
			shows that Sobolev regularity admits an equivalent Hilbertian formulation entirely determined by multiscale texture energies. Thus, a construction originally motivated by texture analysis leads naturally to the classical Sobolev scale, providing a rigorous analytical foundation for the texture operator and revealing its intrinsic connection with one of the fundamental scales of modern functional analysis.
			
			Beyond its theoretical interest, this framework naturally suggests new variational models, inverse problems, and texture-based regularization methods built directly upon the texture operator.
			
	\section{Conclusions and open problems}\label{sec:conclusions}
	
		In this work we have established the mathematical foundations of the pointwise multiscale texture operator
		\begin{equation*}
			\tau_{z,u}(\sigma) = \frac{\partial}{\partial\sigma}(u*\gamma_\sigma)(z),
		\end{equation*}
		which provides a quantitative measure of the persistence of local oscillatory structures under Gaussian scale evolution. Starting from its original definition, we have developed a complete analytical framework describing its structural, spectral and multiscale properties.
		
		Our analysis shows that the texture operator is a linear, infinitely smoothing operator admitting a simple spectral representation as a band-pass filter generated by the scale derivative of the Gaussian kernel. This characterization leads naturally to an $r$-adic discretization, whose associated difference-of-Gaussians family possesses the characteristic dilation, stability and reconstruction properties of multiscale wavelet-type systems.
		
		Finally, we have shown that the multiscale decomposition generated by the texture operator induces a natural family of function spaces. Although these spaces originate from the notion of texture persistence rather than from classical harmonic analysis, they provide an intrinsic characterization of Besov and Sobolev regularity in terms of multiscale texture persistence.
	
		\subsection{Summary of contributions}
	
			The main contributions of this work may be summarized as follows.
			
			\paragraph{The analytical foundations of the multiscale texture operator.}
			
			We have developed the first systematic mathematical study of the pointwise multiscale texture operator. Its boundedness, regularizing properties and scale-space behaviour have been established, showing that the operator is well posed on the classical Lebesgue and Sobolev scales and providing a rigorous justification for its use as a multiscale texture descriptor.
			
			\paragraph{A spectral interpretation of texture persistence.}
			
			The scale derivative of the Gaussian kernel has been shown to possess a band-pass spectral structure whose characteristic frequency is inversely proportional to the observation scale. This establishes a precise connection between texture persistence under Gaussian diffusion and frequency localization, explaining why the operator selectively captures oscillatory structures of a prescribed characteristic size.
			
			\paragraph{A natural multiscale discrete model.}
			
			The continuous theory has been discretized by means of an $r$-adic difference-of-Gaussians decomposition. The resulting family satisfies the characteristic dilation, stability and reconstruction properties of wavelet-type systems, providing an intrinsic multiscale representation associated with the texture operator rather than with an \emph{a priori} frequency partition.
			
			\paragraph{Texture-induced function spaces.}
			
			The multiscale decomposition naturally gives rise to a family of function spaces whose norms are defined directly from the texture coefficients. These spaces are therefore intrinsic to the persistence operator itself and provide a functional framework in which regularity is measured through the distribution of texture energy across scales.
			
			\paragraph{A functional-analytic interpretation of texture persistence.}
			
			Finally, we have shown that the persistence of texture under Gaussian scale evolution provides an intrinsic characterization of classical Besov regularity and, in the Hilbertian case, of Sobolev regularity. Thus, the classical scales of functional analysis admit an equivalent realization entirely in terms of multiscale texture energies.

		\subsection{Open mathematical questions}
			
			Although the present work establishes the analytical foundations of the multiscale texture operator, many mathematical questions remain open.
			
			\paragraph{Intrinsic characterization of texture regularity.}
			
			The identification of the texture spaces with the classical Besov scale shows that the operator captures a well-established notion of regularity. Nevertheless, it would be desirable to obtain an intrinsic characterization of texture regularity directly in terms of the persistence operator itself, without relying on the Littlewood--Paley machinery. Such a characterization could provide a more geometric interpretation of multiscale texture.
			
			\paragraph{Alternative scale-space evolutions.}
			
			Throughout this work the Gaussian kernel has played a central role because of its intimate connection with the heat equation. It is natural to ask whether analogous constructions can be developed for other scale-space evolutions, such as fractional diffusion, anisotropic diffusion or nonlinear diffusion processes. This could lead to new classes of texture operators adapted to different physical or geometric settings.
			
			\paragraph{Directional and anisotropic texture operators.}
			
			The present operator is isotropic by construction. Many natural textures, however, exhibit preferred orientations or anisotropic structures. Replacing the isotropic Gaussian kernel by anisotropic kernels or directional diffusion processes may provide texture operators capable of distinguishing orientation as well as scale. The corresponding functional framework remains to be developed.
			
			\paragraph{Adaptive multiscale texture analysis.}
			
			In many applications the characteristic scale of the texture varies across the image. This suggests considering adaptive versions of the operator in which the observation scale depends on the spatial position or evolves according to local image features. The mathematical analysis of such nonlinear and spatially adaptive operators is still completely open.
			
			\paragraph{Variational models based on texture energies.}
			
			One of the motivations for introducing a rigorous functional framework is the construction of variational models in which texture persistence plays the role traditionally assigned to gradients or wavelet coefficients. The texture norms introduced in this work provide a natural starting point for studying regularization, inverse problems and multiscale image decomposition driven directly by texture.
			
			\paragraph{Evolution equations driven by texture fields.}
			
			Another promising direction consists in studying partial differential equations whose coefficients are determined by texture fields extracted from a reference image. Examples include diffusion operators of the form
			\begin{equation*}
				-\nabla\!\cdot(D(T)\nabla u),
			\end{equation*}
			or reaction--diffusion systems evolving on media whose heterogeneity is encoded by a multiscale texture field $T$. Understanding how the geometry of the texture influences localization phenomena, spectral properties and pattern formation appears to be a particularly interesting direction for future research.
			
			From a broader perspective, the present work shows that texture persistence under Gaussian scale evolution is not only an effective descriptor for image analysis, but also gives rise naturally to classical structures of harmonic analysis and function space theory.
	
		\subsection{Perspectives for applications}
		
			Although the present work is primarily theoretical, its motivation originates in multiscale texture analysis and image processing. The analytical framework developed here provides a rigorous mathematical basis for future models in which texture persistence plays a central role.
			
			A first natural direction concerns variational methods. Since the texture operator measures the persistence of oscillatory structures across scales, the associated multiscale texture norms provide an alternative to classical gradient- or wavelet-based regularization terms. This opens the possibility of developing new variational models for image restoration, denoising, segmentation and decomposition in which texture, rather than intensity or gradient magnitude, becomes the primary quantity being controlled.
			
			The $r$-adic difference-of-Gaussians decomposition also leads to an efficient computational implementation. The texture coefficients are obtained by Gaussian convolutions at geometrically distributed scales, making the resulting multiscale representation naturally compatible with fast FFT-based algorithms and standard Gaussian pyramids. Consequently, the proposed framework may be incorporated into practical texture analysis pipelines without requiring fundamentally new numerical techniques.
			
			Beyond image processing, the notion of texture persistence may also prove useful in the analysis of multiscale spatial data arising in other applications. Examples include biomedical imaging, remote sensing, materials science and biological pattern analysis, where oscillatory structures often carry relevant geometric information across several spatial scales.
			
			Finally, the functional framework introduced in this paper provides a natural starting point for the development of new mathematical models driven by texture rather than by derivatives alone. We expect that the combination of multiscale texture operators, variational methods and partial differential equations will constitute a promising direction for future research, both from the theoretical and the applied points of view.

	\bibliographystyle{amsplain}

\end{document}